\documentclass[11pt,a4paper]{amsart}

\usepackage{amsmath}
\usepackage{amsthm}
\usepackage{enumerate}
\usepackage{amssymb}
\usepackage{graphicx}
\usepackage[all]{xy}
\usepackage{hyperref}
\usepackage{aliascnt}
\usepackage{stmaryrd}
\usepackage{braket}
\usepackage{xcolor}
\usepackage{mdframed}
\usepackage{mathtools}

\newtheorem{lma}{Lemma}[section]

\newaliascnt{thmCt}{lma}
\newtheorem{thm}[thmCt]{Theorem}
\aliascntresetthe{thmCt}

\newaliascnt{corCt}{lma}
\newtheorem{cor}[corCt]{Corollary}
\aliascntresetthe{corCt}

\newaliascnt{prpCt}{lma}
\newtheorem{prp}[prpCt]{Proposition}
\aliascntresetthe{prpCt}

\newtheorem*{thm*}{Theorem}
\newtheorem*{cor*}{Corollary}
\newtheorem*{prp*}{Proposition}

\theoremstyle{definition}
\newtheorem{pgr}[lma]{}

\newaliascnt{dfnCt}{lma}
\newtheorem{dfn}[dfnCt]{Definition}
\aliascntresetthe{dfnCt}

\newaliascnt{rmkCt}{lma}
\newtheorem{rmk}[rmkCt]{Remark}
\aliascntresetthe{rmkCt}

\newaliascnt{rmksCt}{lma}
\newtheorem{rmks}[rmksCt]{Remarks}
\aliascntresetthe{rmksCt}

\newaliascnt{ntnCt}{lma}

\aliascntresetthe{ntnCt}

\newaliascnt{qstCt}{lma}

\aliascntresetthe{qstCt}

\newaliascnt{prblCt}{lma}
\newtheorem{prbl}[prblCt]{Problem}
\aliascntresetthe{prblCt}

\newaliascnt{exaCt}{lma}
\newtheorem{exa}[exaCt]{Example}
\aliascntresetthe{exaCt}

 \newcommand{\V}{\mathrm{V}}
 
 \newcommand{\N}{\mathbb{N}}

 \newcommand{\Z}{\mathbb{Z}}
 \newcommand{\K}{\mathrm{K}}
 \newcommand{\KK}{\mathcal{K}}

\newcommand{\CatCu}{\ensuremath{\mathrm{Cu}}}

\newcommand{\ZZ}{\mathcal Z}

\newcommand{\CatPreW}{\mathrm{PreW}}

\newcommand{\ca}{C*-algebra}

\newcommand{\axiomO}[1]{(O#1)}

\DeclareMathOperator{\Gr}{Gr}

\DeclareMathOperator{\Cu}{Cu}

\newcommand{\Mau}{\mathrm{Gr}(M)_{\mathrm{au}}^+}

\makeatletter

\newcommand{\Rmnum}[1]{\expandafter\@slowromancap\romannumeral #1@}
\makeatother

\newcommand{\beginEnumStatements}{\begin{enumerate}[\quad(1)]}
\newcommand{\beginEnumConditions}{\begin{enumerate}[\quad(i)]}

\newcommand{\beginTodo}{\begin{mdframed}[backgroundcolor = cyan, %
outerlinewidth = 2,%
leftmargin = 0,%
rightmargin = 0,%
innertopmargin = 0,
]}
\newcommand{\stopTodo}{\end{mdframed}}

\begin{document}

\title{Perforation conditions and almost algebraic order in Cuntz semigroups}

\author{Ramon Antoine}
\author{Francesc Perera}
\author{Henning Petkza}
\date{October 17, 2016}

\address{Ramon Antoine, Francesc Perera,
Departament de Matem\`{a}tiques \\
Universitat Aut\`{o}noma de Barcelona \\
08193 Bellaterra, Barcelona, Spain
}
\email[]{ramon@mat.uab.cat}
\email[]{perera@mat.uab.cat}

\address{Henning Petzka, Mathematisches Institut,
Universit\"at M\"unster,
Einsteinstra{\ss}e 62, 48149 M\"{u}nster, Germany}

\email[]{petzka@uni-muenster.de}

\thanks{R.A. and F.P. were partially supported by grants DGI MICIIN MTM2011-28992-C02-01 and MINECO MTM2014-53644-P; H.P. was supported by the DFG (SFB 878). Part of the results in this paper were obtained while the second author was attending the program \emph{Classification of operator algebras: complexity, rigidity and dynamics} at the Mittag-Leffler Institute, January-April 2016.}

\subjclass[2010]
{Primary
06B35, 
15A69, 
46L05. 
46L35, 
46L80, 
Secondary
18B35, 
19K14, 
46L06, 
}

\keywords{Cuntz semigroup, tensor product, $C^*$-algebra}

\begin{abstract}
For a \ca{} $A$, it is an important problem to determine the Cuntz semigroup $\Cu(A\otimes\ZZ)$ in terms of $\Cu(A)$. We approach this problem from the point of view of semigroup tensor products in the category of abstract Cuntz semigroups, by analysing the passage of significant properties from $\Cu(A)$ to $\Cu(A)\otimes_\CatCu\Cu(\ZZ)$. We describe the effect of the natural map $\Cu(A)\to\Cu(A)\otimes_\CatCu\Cu(\ZZ)$ in the order of $\Cu(A)$, and show that, if $A$ has real rank zero and no elementary subquotients, $\Cu(A)\otimes_\CatCu\Cu(\ZZ)$ enjoys the corresponding property of having a dense set of (equivalence classes of) projections. In the simple, nonelementary, real rank zero and stable rank one situation, our investigations lead us to identify almost unperforation for projections with the fact that tensoring with $\ZZ$ is inert at the level of the Cuntz semigroup. 
\end{abstract}

\maketitle

\section*{Introduction}
Conditions of perforation in the Cuntz semigroup of a \ca{} and tensorial absorption of the \ca{} $\mathcal{Z}$, constructed by Jiang and Su \cite{JiangSu}, play an essential part in the classification theory of \ca{s}. This is most prominently witnessed by A. Toms' counterexample in \cite{Toms08}, where the author exhibits a simple, separable, unital, nuclear, stably finite \ca{} $A$ that has the same Elliott invariant as $A\otimes U$ for any UHF-algebra $U$. Yet, the said algebras are not isomorphic since the Cuntz semigroup of $A$ fails to be almost unperforated. In particular, $A$ does not absorb the Jiang-Su algebra $\mathcal Z$ tensorially (a condition referred to as $\mathcal{Z}$-stability).

The exact relationship between almost unperforation and $\mathcal{Z}$-stability is one of the ingredients in the well-known Toms-Winter conjecture, that treats the situation of simple, separable, unital, nuclear, and infinite dimensional \ca{s} (see, for example, \cite[Remark 3.5]{TomWin09}, \cite[Conjecture 0.1]{Win10}, and \cite[Conjecture 9.3]{WinZac10}). For these \ca{s}, part of this conjecture predicts that almost unperforation in the Cuntz semigroup and tensorial absorption of the Jiang-Su algebra are equivalent conditions. Following the work of many authors, this equivalence is now known to hold in the case where the tracial simplex is a Bauer simplex whose extreme boundary has finite covering dimension. More concretely, that $\ZZ$-stability implies almost unperforation holds in full generality (see \cite{Rordam04}). The converse follows from the (independent) results in \cite{KirRorCentralSeq}, \cite{SatoTraceSpacesOfSimple} and \cite{TomWhiWin12} (see also \cite{MatSat12}).
Moreover, in some particular cases, a computation of the Cuntz semigroup of $A\otimes \mathcal{Z}$ has been carried out (see e.g. \cite{BroTomIMRN07,BroPerTom08,EllRobSan11}). 

Albeit possessing this knowledge on the structure of the Cuntz semigroup of a $\mathcal{Z}$-stable algebra, the computation of $\Cu(A\otimes \mathcal{Z})$, given $\Cu(A)$, remains open in the general case. Our study concerns structural properties of $\Cu(A\otimes \mathcal{Z})$ for \ca{s} with low rank. We restrict our investigation to \ca{s} with real rank zero, as for such algebras the Cuntz semigroup is algebraic (that is, the subsemigroup of the so called compact elements is dense; see \cite{APT14}). In fact, as shown in \cite{CowEllIvanCrelle08}, if $A$ is a \ca{} with stable rank one, then $A$ has real rank zero precisely when $\Cu(A)$ is algebraic. Towards a successful computation of $\Cu(A\otimes \mathcal{Z})$ when $A$ has real rank zero, the following question naturally arises: If $A$ is a \ca{} of real rank zero (with no elementary subquotient), is then $A \otimes \mathcal{Z}$ of real rank zero? In terms of the Cuntz semigroup, a related question is whether $A$ having real rank zero implies that $\Cu(A\otimes\ZZ)$ is algebraic.

We note that this question has a positive answer in the simple case. Indeed, let $A$ be a unital, simple, exact, non-type I \ca{}. Then $A\otimes\ZZ$ is either purely infinite (simple) or else it has stable rank one, by the results in \cite{Rordam04}. In the first case, it is well known that $A\otimes\ZZ$ has real rank zero. In the second case, if $A$ has further real rank zero, then the image of $\K_0(A)$ in the space $\mathrm{Aff}(T(A))$ of real valued, affine continuous functions on the trace simplex is dense (see \cite{AraGoodearlOMearaPardo98} and \cite{ParTAMS98}). Since, as groups, $\K_0(A)\cong\K_0(A\otimes\ZZ)$ and also $T(A)\cong T(A\otimes\ZZ)$, this density condition is also satisfied by $A\otimes\ZZ$. Then \cite[Theorem 7.2]{Rordam04} yields that $A\otimes\ZZ$ has real rank zero. 

A natural approach to answering the question of whether the real rank of $A\otimes\mathcal{Z}$ is zero if $A$ has real rank zero and no elementary subquotients, consists of studying it in the abstract setting of semigroups in the category $\CatCu$, as defined in \cite{CowEllIvanCrelle08} (see below for the precise definition). This can be carried out in (at least) two ways. The first one consists of identifying the subsemigroup of compact elements in the Cuntz semigroup of $A\otimes\ZZ$. To this end, we shall mostly focus on (residually) stably finite \ca{s} $A$, as in this setting the said subsemigroup agrees with $\V(A)$, the Murray-von Neumann semigroup of projections. In the simple case, the computation of $\V(A\otimes\ZZ)$ is carried out in \cite{GongJiangSu00} in the form of an almost unperforated subsemigroup of the Grothendieck group of $\V(A)$. This, of course, uses (stable) cancellation of projections of $A \otimes \mathcal{Z}$ in an essential way, which is possible because the stable rank of a finite, simple $\mathcal{Z}$-stable \ca{} is one (\cite{Rordam04}). In the non-simple setting, however, we only know that $\V(A \otimes \mathcal{Z})$ is separative (see \cite{Jiang97}). To compensate, we bring in an additional and sufficient condition on the \ca{} $A$ that ensures cancellation for $A \otimes \mathcal{Z}$. This condition is termed cancellation into ideals, which holds vacuously in the simple case and potentially more widely in general. Under this additional assumption and assuming residually stably finiteness, we can then extend the computation of $\V(A\otimes \mathcal{Z})$ from \cite{GongJiangSu00} to the non-simple setting.

A second way of approaching the question of when the real rank of $A\otimes\ZZ$ is zero
consists of studying the relationship of $\Cu(A\otimes\ZZ)$ with $\Cu(A)$ and $\Cu(\ZZ)$. The first two authors, in collaboration with H. Thiel, initiated in \cite{APT14} the investigation of the Cuntz semigroup of tensor products of \ca{s} by introducing and studying the tensor product $\otimes_\CatCu$ in the category $\CatCu$. Denoting $Z=\Cu(\ZZ)$,  a verification of a semigroup analogue of part of the Toms-Winter conjecture is shown to hold in \cite[Theorem 7.3.11]{APT14}: A semigroup $S$ in $\CatCu$ is almost unperforated and almost divisible if and only if $S\cong S\otimes_\CatCu Z$. It follows that, if the Toms-Winter conjecture holds true, then the condition $\Cu(A)\cong\Cu(A)\otimes_\CatCu Z$ is equivalent to $A\cong A\otimes\mathcal{Z}$.

Given \ca{s} $A$ and $B$, of which at least one is nuclear, it was proved in \cite{APT14} that there is a natural map $\Cu(A)\otimes_\CatCu \Cu(B)\to\Cu(A\otimes B)$. This map was shown to be an isomorphism when one of the algebras is either an AF-algebra, or a nuclear $\mathcal{O}_\infty$-stable algebra (\cite[Sections 6,7]{APT14}). It remains an open problem to decide when the above map is an isomorphism. (See \cite[Problem 6.4.11]{APT14}; it is known this is not the case in general, as observed in \cite[6.4.12]{APT14}.) Still, it is natural to consider the following question: Given a \ca{} $A$ of real rank zero with no elementary subquotient, is $\Cu(A)\otimes_\CatCu\Cu(\ZZ)$ also algebraic? This can be phrased in the more general setting of the category $\CatCu$ as follows: If $S\in\CatCu$ is algebraic, when does it follow that $S\otimes_\CatCu Z$ is also algebraic? We show here that, if $S$ is algebraic and almost divisible, then $S\otimes_\CatCu Z$ is algebraic, suggesting that indeed under reasonable assumptions, the real rank of $A\otimes \mathcal{Z}$ is zero whenever the real rank of $A$ is zero. 

In the abstract setting, this question is related to two problems posed in \cite{APT14}. The first of these problems (\cite[Problem 7.3.13]{APT14}) focuses on the semigroup analogue of the natural embedding $j\colon A\to A\otimes\ZZ$, which induces a map $\Cu(j)\colon\Cu(A)\to\Cu(A\otimes\ZZ)$. This factorises as
\[
\Cu(A)\to\Cu(A)\otimes_\CatCu\Cu(\ZZ)\to\Cu(A\otimes\ZZ)\,,
\]
and thus it is natural to ask, at the semigroup level, in which way $\Cu(A)$ sits in $\Cu(A)\otimes_\CatCu\Cu(\ZZ)$. More concretely, given $S\in\CatCu$, is it possible to characterize when $x\otimes 1\leq y\otimes 1$ in $S\otimes_\CatCu Z$? In the same vein, a natural subproblem asks to characterize how $\V(A)$ sits in $\V(A)\otimes 1$, seen as a subsemigroup of $\Cu(A)\otimes_\CatCu\Cu(\ZZ)$.

The second problem (\cite[7.3.14]{APT14}) aims at measuring how far is $\Cu(A)\otimes_\CatCu\Cu(\ZZ)$ from being the Cuntz semigroup of a \ca{}. It is known that $\Cu(A)$ always satisfies two additional axioms than the ones originally used to define the category $\CatCu$. The first one is referred to as the axiom of almost algebraic order, or also axiom \axiomO{5}, and indeed can be thought of as a deformation of the algebraic order (see \cite{RorWinCrelle10}). The second one is referred to the axiom of almost Riesz decomposition, or also axiom \axiomO{6}, and it is a version of the Riesz decomposition property in an ordered semigroup; see \cite{Rob13Cone}. (See below for the precise formulations of these axioms.) It is also known that, if $A$ has stable rank one, $\Cu(A)$ has the so called property of weak cancellation.  Thus specifically, if $S\in\CatCu$, we ask whether axioms \axiomO{5}, \axiomO{6} or weak cancellation pass from $S$ to $S\otimes_\CatCu Z$.

Building on the work in \cite{APT14}, we provide partial answers to the above questions. For an algebraic semigroup $S\in\CatCu$ satisfying \axiomO{5}, \axiomO{6}, weak cancellation and almost divisibility, we completely characterize when $x\otimes 1\leq y\otimes 1$ in $S\otimes_\CatCu Z$ in terms of comparison of two successive (additive) powers of $x$ and $y$. This relies on a similar characterization we obtain in purely algebraic tensor products. 

Our results on the last question above yield a connection of almost unperforation with axiom \axiomO{5}: If $S$ is an algebraic, almost divisible $\CatCu$-semigroup satisfying \axiomO{5}, and such that the subsemigroup of compact elements is separative, then $S\otimes_\CatCu Z$ satisfies \axiomO{5} if and only if $S$ is almost unperforated. For a simple, non-type I \ca{} $A$ of real rank zero and stable rank one, this results in the equivalence of $\Cu(A)$ being almost unperforated with $\Cu(A)\otimes_{\CatCu} \Cu(\ZZ)$ satisfying \axiomO{5}. In turn, this is also equivalent to having $\Cu(A)\cong\Cu(A\otimes\ZZ)$. If $A$ has, further, locally finite nuclear dimension, the latter is equivalent to $A\cong A\otimes\ZZ$, by \cite[Corollary 7.4]{Win12}.

The paper is organized as follows. Sections~\ref{sec:possem} and \ref{sec:divisibility} are devoted to discuss cancellation, divisibility and perforation properties in positively ordered semigroups, with some applications to $\CatCu$-semigroups and \ca{s}. We move on in Section~\ref{sec:cones} to study almost unperforated semigroups and, in particular, embeddings of algebraically ordered semigroups into almost unperforated cones. In Section~\ref{sec:OrderTensor} we discuss the order of tensor products of algebraically ordered semigroups with the Cuntz semigroup of $\mathcal{Z}$. Finally, Section \ref{sec:ApplicationsToCu} contains the applications of the previous sections to $\CatCu$-semigroups and Cuntz semigroups of \ca{s} of real rank zero, including answers to the above questions.

\section{Preliminaries}
\label{sec:Preliminaries}

The purpose of this section is to recall various definitions concerning positively ordered semigroups that we shall be using throughout.

\begin{pgr}[Positively ordered semigroups]
\label{pgr:pom}
All semigroups in this paper are commutative, written additively, and have a neutral element, that we shall denote by 
$0$.

Recall that a positively ordered semigroup is a semigroup $M$ with a translation invariant partial order $\leq$ such that $0\leq x$ for all $x\in M$ (see \cite{WehInj92} and also \cite[B.2.1]{APT14}). Notice that such a semigroup will always be \emph{conical}, that is, $x+y=0$ if and only if $x=y=0$.

If $M$ is just a semigroup, then $M$ can be equipped with the \emph{algebraic ordering}: $x\leq_{\mathrm{alg}}y$ provided there is $z\in M$ with $x+z=y$. Note that $0\leq_{\mathrm{alg}}x$ for all $x\in M$ and that $\leq_{\mathrm{alg}}$ is clearly translation invariant, though the algebraic ordering might not necessarily be a partial order. If $M$ is \emph{finite} in the sense that $x+y=x$ implies $y=0$, then $(M,\leq_{\mathrm{alg}})$ is also partially ordered. Therefore the projection semigroup $\V(A)$ of a stably finite \ca{} $A$ is, equipped with the order induced from Murray-von Neumann comparison of projections, a positively ordered semigroup. Since it will be clear from the context, we will not generally use the notation $\leq_{\mathrm{alg}}$ to refer to the algebraic order. Notice that the order in any positively ordered semigroup always extends $\leq_\mathrm{alg}$.
\end{pgr}

\begin{pgr}[order ideals and quotients]
\label{pgr:ideal}
Recall that a subsemigroup $I$ of a positively ordered semigroup $M$ is an \emph{order ideal} provided that $x\leq y$ with $y\in I$ entails that $x$ is an element of $I$. Given $x\in M$, the order ideal generated by $x$ will be denoted by $I(x)$ and will be termed a \emph{principal ideal}. Observe that $I(x)=\{y\in M\mid y\leq nx\text{ for some } n\in\N \}$. As is customary, $M$ is said to be \emph{simple} in case the only order ideals are the trivial ones.

In case $M$ is given the algebraic order, then a subsemigroup $I$ is an order ideal provided $x+y$ is an element of $I$ precisely when $x$ and $y$ are. The ideal generated by $x$ can in this situation be described as $I(x)=\{y\in M\mid y+z=nx\text{ for some } z\in M\text{ and }n\in\N \}$.

Recall that the quotient $M/I$ of a positively ordered semigroup $M$ by an order ideal $I$ is defined as $M/\!\!\sim$ where $\sim$ is the congruence given by: $x\sim y$ if and only if $x+v=y+w$ with $v, w\in I$. Denote the classes of $M/I$ by $[x]$ for $x\in M$. This new semigroup can be equipped with the translation invariant preorder $[x]\leq [y]$ provided $x\leq y+v$, for $v\in I$. If $M$ is algebraically ordered, then $M/I$ is always conical, but it need not be partially ordered. For general, positively ordered semigroups, the notion of finiteness may be adjusted to say $x+y\leq y$ only when $x=0$. For conical, algebraically ordered semigroups, this agrees with our definition of finiteness given in \ref{pgr:pom}.

We shall often require this stronger finiteness condition and then say that $M$ is \emph{strongly finite} provided $M/I$ is finite for all order ideals $I$. (In some contexts, this condition is referred to as $M$ being residually finite.) If $M$ is algebraically ordered and strongly finite, then all of its quotients are partially ordered. 
\end{pgr}

The following elementary fact will be repeatedly used in the sequel, hence we record it for future reference.

\begin{lma}\label{lma:SameIdeal}
Let $M$ be a strongly finite, positively ordered semigroup.
If for $x,y,z\in M$,
$$x+z \leq y+z,$$
then $I(x)\subseteq I(y)$.
\end{lma}

\begin{proof}
Let $I=I(y)$ be the order ideal generated by $y$. Then in $M/I$ we have $[z]+[x]\leq [z]$. By the finiteness assumption on $M$, the element $x$ lies in $I$.
\end{proof}

\begin{pgr}[The category $\CatCu$]
Given elements $x,y$ in a positively ordered semigroup $S$, we say that $x$ is \emph{compactly contained} in $y$, in symbols $x\ll y$, if whenever $(z_n)_n$ is an increasing sequence in $S$ for which the supremum exists and which satisfies $y\leq \sup_n z_n$, then there exists $k\in\mathbb{N}$ with $x\leq z_k$.

As mentioned in the introduction, the category $\CatCu$ was introduced in \cite{CowEllIvanCrelle08}. The objects are positively ordered semigroups $S$ subject to four axioms: 
\begin{itemize}
\item[\axiomO{1}] Every increasing sequence in $S$ has a supremum; 
\item[\axiomO{2}] Any element $x\in S$ is the supremum of a sequence $(x_n)$ such that $x_n\ll x_{n+1}$ for all $n$; \item[\axiomO{3}] The relation $\ll$ is compatible with addition;
\item[\axiomO{4}] Suprema and addition are compatible.
\end{itemize}

We refer to an object $S\in\CatCu$ as an (abstract) $\CatCu$-semigroup. Recall that the \emph{Cuntz semigroup} $\Cu(A)$ of a \ca{} $A$ is defined as 
\[
\Cu(A)=(A\otimes K)_+/\!\!\sim
\]
where $\sim$ is the antisymmetrization of the subequivalence defined by $a\precsim b$ if and only if $a=\lim\limits_{n} x_nbx_n^*$ for some sequence $(x_n)_n$. It was shown in \cite{CowEllIvanCrelle08} that $\Cu(A)\in\CatCu$ for every \ca{} $A$.

Ideals of $\CatCu$-semigroups are order-ideals as in \ref{pgr:ideal} that are further closed under suprema of increasing sequences. The quotient of a $\CatCu$-semigroup $S$ by an order ideal $I$ is defined also as in \ref{pgr:ideal}, except that the order is antisymmetrized, and thus $S/I$ is always positively ordered. In fact, $S/I$ becomes an object in $\CatCu$ (see \cite[Lemma 5.1.2]{APT14}).

Recall from \cite[5.2.2]{APT14} that a $\CatCu$-semigroup is \emph{stably finite} provided $x+y=y$ only when $x=0$ for any $y$ such that there is $\tilde{y}\in S$ with $y\ll \tilde{y}$. We will then say here that $S$ is strongly finite when $S/I$ is stably finite for any ideal $I$ of $S$.
\end{pgr}

\begin{pgr}[Preminimality]
\label{sct:premin}
Not all semigroups in this paper will be algebraically ordered, namely the Cuntz semigroup of a \ca{} rarely is. The following concept is then useful:

A positively ordered semigroup $M$ is \emph{preminimal} if, whenever $x,y,v\in M$ satisfy $x+v\leq y+v$ and $v\leq w$, for $w\in M$, then $x+w\leq y+w$ (see \cite[Definition 1.2]{WehComm}). Of course, any algebraically ordered semigroup is automatically preminimal.

For $\CatCu$-semigroups (even those that come from \ca{s}), this will be a rare condition to be satisfied. For example, $Z:=\Cu(\ZZ)$ is not preminimal. To check this, identify $Z$ with $\N_0\sqcup (0,\infty]$, and denote as $1'$ the unit in $(0,\infty]$. We now have $1'\leq 1$ and also $1+1'\leq 1'+1'$, but $2\not\leq 1'+1$.

The following notion, close to preminimality, is much more frequent. Let us say that a $\CatCu$-semigroup $S$ is \emph{weakly preminimal} if whenever $x+z\ll y+z$ and $z\leq w$, we have $x+w\leq y+w$.

Recall that a $\CatCu$-semigroup $S$ satisfies \axiomO{5} (the axiom of \emph{almost algebraic order}) if, whenever $x+z\leq y$ in $S$ and $x'\ll x$, $z'\ll z$, then there is $w\in S$ with $z'\leq w$ and $x'+w\leq y\leq x+z$. For any \ca{} $A$, the semigroup $\Cu(A)$ satisfies \axiomO{5} (see \cite[Definition 4.1, Proposition 4.7]{APT14}, and also \cite[Lemma 7.1]{RorWinCrelle10}).

As shown in \cite[Lemma 5.6.7]{APT14}, any $\CatCu$-semigroup satisfying \axiomO{5} is weakly preminimal. In particular, $\Cu(A)$ is weakly preminimal for any \ca{} $A$.

The exact relationship between these two concepts is clarified if we consider algebraic $\CatCu$-semigroups.
In order to make the connection precise, recall that an element $x$ in a $\CatCu$-semigroup is \emph{compact} provided $x\ll x$. The subsemigroup of compact elements of $S$ is denoted by $S_c$. Recall from \cite{APT14} (see also \cite{Compendium}) that a $\CatCu$-semigroup is \emph{algebraic} provided that every element of $S$ is the supremum of an increasing sequence of compact elements. If $A$ is a stably finite \ca{}, $\V(A)$ can be identified with the subsemigroup of compact elements in $\Cu(A)$ (see, for example, \cite{BrownCiuperca09}). 
\end{pgr}

\begin{lma}
\label{lma:preminimal} Let $S$ be an algebraic $\CatCu$-semigroup. Then $S$ is weakly preminimal if and only if $S_c$ is preminimal. 
\end{lma}
\begin{proof}
We need only show the ``if'' condition. Thus assume $x+z\ll y+z$ and $z\leq w$. Choose $z'\ll z'\ll z$ such that $x+z'\ll y+z'$. Write $x$ and $y$ as suprema of sequences $(x_n)$ and $(y_n)$ of compact elements, respectively. We have $x_n+z'\leq y_n+z'$ for sufficiently large $n$. Choose an increasing sequence $(w_m)$ of compact elements with suprema $w$. Since $z'\ll w$, there is $m_0$ such that $z'\leq w_m$ whenever $m\geq m_0$. Apply preminimality of $S_c$ to conclude $x_n+w_n\leq y_n+w_n$ for suffciently large $n$. Taking suprema we obtain $x+w\leq y+w$, as desired.
\end{proof}

\section{Cancellation conditions}
\label{sec:possem}

In this section we analyse various notions of cancellation for semigroups associated to \ca{s}, and introduce the much weaker condition of cancellation into ideals.

\begin{pgr}[Separativity]
\label{sct:separat}
There is no doubt that cancellation is a very useful condition that helps analysing the structure of a semigroup. For $\V(A)$, it is automatic when $A$ has stable rank one. For general posivitely ordered semigroups, full cancellation might not be present, although there are weaker forms that we define below. We also need to notice that, in this general context, $M$ is \emph{cancellative} if $x+z=y+z$ implies $x=y$ and $M$ is \emph{order-cancellative} if $x+z\leq y+z$ implies $x\leq y$. For algebraically ordered semigroups these two notions of cancellation agree.

A semigroup $M$ is \emph{separative} if $2x=x+y=2y\implies x=y$, for $x,y\in M$ (see, for example, \cite{CliPreVolume61}). In \cite{AraGoodearlOMearaPardo98} separativity is shown to be equivalent to a number of properties for algebraically ordered semigroups, one of them being that, if $x,y,z$ are elements in $M$ such that $x+z=y+z$ and $z$ belongs to $I(x), I(y)$, then $x=y$.

In case $M$ is positively ordered, then we say that $M$ is \emph{order-separative} if $M$ is preminimal, $M$ is separative as a semigroup, and furthermore $x+y\leq 2y\implies x\leq y$. (See \cite[Definition 1.2 and Theorem 1.4]{WehComm}.)

We consider now a version of this concept, that somewhat sits between strong separativity (in the sense of \cite{Moreira}) and separativity:

We say that $M$ is \emph{nearly separative} if $M$ is preminimal and $2x\leq x+y\leq 2y\implies x\leq y$. (This notion was termed weak separativity in \cite{APT14}, but as we will see below it is not really a weakening of the concept of separative cancellation.)
 
Observe that a nearly separative, positively ordered, semigroup is necessarily separative. In particular, for a partially ordered semigroup, near unperforation as will be defined below implies separativity.
\end{pgr}

\begin{lma}
\label{lma:Wsep}
Let $M$ be a positively ordered semigroup. Consider the following conditions:
\begin{enumerate}[{\rm (i)}]
\item $M$ is nearly separative.
\item $x+z\leq y+z$ with $z\in I(x)$ and $z\in I(y)$ implies $x\leq y$.
\item $x+2z\leq y+2z$ implies $x+z\leq y+z$.
\end{enumerate}
Then {\rm (i)} $\iff$ {\rm (ii)} $\implies$ {\rm (iii)}. 
\end{lma}
\begin{proof}
Assume (i). To prove (iii), assume $x+2z\leq y+2z$. It follows easily that $2(x+z)\leq (x+z)+(y+z)\leq 2(y+z)$, whence (i) implies $x+z\leq y+z$.

To show that (ii) holds also assuming (i), we use that already (iii) above holds. Suppose $x+z\leq y+z$ with $z\in I(x)$, $z\in I(y)$. There is $k\in \N$, that we can take to be a power of $2$, such that $z\leq kx$ and also $z\leq ky$. By preminimality, it follows that $x+kx\leq y+kx$ and hence, applying condition (iii) repeatedly, we obtain $2x\leq x+y$. Likewise, we obtain $x+y\leq 2y$. Now condition (i) implies $x\leq y$.

That (ii) implies (i) is easy. We only need to notice that if $2x\leq x+y\leq 2y$, then $x+x\leq x+y$ with $x\leq 2x\leq 2y$. It is also easy to verify that (ii) implies $M$ is preminimal.
\end{proof}

\begin{rmk}
\label{rmk:sepimpliesnear}
The proof of \autoref{lma:Wsep} shows that (i)$\implies$(ii) for any order-separative semigroup.

Observe also that an (order-)separative semigroup is, when equipped with the algebraic order, automatically nearly separative if it is finite. (This follows easily from condition (ii) in the above Lemma.) Thus, this notion is of relevance for positively ordered semigroups whose order is not algebraic.
\end{rmk}

Recall that a $\CatCu$-semigroup $S$ has \emph{weak cancellation} provided $x+z\ll y+z$, for $x, y, z\in S$, implies $x\ll y$. This is always satisfied for $\Cu(A)$ whenever $A$ has stable rank one (\cite[Theorem 4.3]{RorWinCrelle10}, and also \cite{RobSanJFA10}). For $\CatCu$-semigroups, it is convenient that the notion of separativity is adjusted as follows: a $\CatCu$-semigroup $S$ is \emph{weakly separative} if, whenever $x+z\ll y+z$ with $x,y,z\in S$ and $z\in I(x),I(y)$, it follows that $x\ll y$. Also, $S\in\CatCu$ will be nearly separative if $S$ is weakly preminimal and $2x\leq x+y\leq 2y$ implies $x\leq y$.

\begin{lma}
\label{lma:WsepCu}
Let $S$ be a $\CatCu$-semigroup. If $S$ is nearly separative then $S$ is weakly separative.
\end{lma}
\begin{proof}
We know that condition (iii) in \autoref{lma:Wsep} holds (using the same proof). If $x+z\ll y+z$ with $z\in I(x), I(y)$, then there is $z'\ll z$ and $y'\ll y$ such that $x+z'\ll y'+z'$ and $z'\leq 2^nx$, $z'\leq 2^ny'$ for some $n$. Arguing as in \autoref{lma:Wsep}, using weak preminimality instead of just preminimality, we obtain $x\leq y'\ll y$.
\end{proof}

\begin{cor}
\label{cor:sep} Let $S$ be an algebraic $\CatCu$-semigroup.
\begin{enumerate}[{\rm (i)}]
\item $S$ is weakly separative if and only if $S_c$ is nearly separative.
\item If $S$ satisfies \axiomO{5}, then $S$ is weakly separative if and only if $S_c$ is separative.
\end{enumerate} 
\end{cor}
\begin{proof}
(i): We already know from \autoref{lma:preminimal} that $S$ is weakly preminimal precisely when $S_c$ is preminimal. If $S$ is weakly separative, it is easy to check that it is also preminimal, and thus $S_c$ is preminimal. It is therefore clear that $S_c$ is nearly separative.

For the converse, assume that $S_c$ is nearly separative and suppose that $x+z\ll y+z$ with $z\in I(x), I(y)$. Arguing as in \autoref{lma:preminimal}, we obtain sequences of compact elements $(x_n)$, $(y_n)$, whose suprema are $x$ and $y$ respectively, and a compact element $z'\leq z$ such that $x_n+z'\leq y_n+z'$ for all $n$. Now, as $z'\ll z$, there is $k$ with $z'\leq kx, ky$, and thus $z'\leq kx_n,ky_n$ for sufficiently large $n$. Since $S_c$ is nearly separative, we obtain $x_n\leq y_n$, and thus $x\leq y$.

(ii): Assuming \axiomO{5}, $S_c$ is algebraically ordered.
\end{proof}

\begin{pgr}[Cancellation into ideals]
\label{pgr:cancintoideals}
We say that a semigroup $M$ has \emph{cancellation into ideals} if, whenever $x,y\in M$ are contained in some order ideal $I$ and $x+z=y+z$ for $z\in M$, then there is $v\in I$ such that $x+v=y+v$. Equivalently, $M$ has cancellation into ideals if, whenever we have $x+z=y+z$ for some $x,y,z\in M$, there exists $v$ in the principal ideal $I(x+y)$ such that $x+v=y+v$.

If $M$ is positively ordered, then $M$ has \emph{order-cancellation into ideals} if, further to the above condition, whenever $x+z\leq y+z$, then $x+v\leq y+v$ for some 
$v\in I(x+y)$.

Finally, we say that $M$ has \emph{strong order-cancellation into ideals} if, whenever $x+z\leq y+z$ with $x\in I(y)$, then there is $v\in I(x)$ such that $x+v\leq y+v$.

Note that any form of cancellation into ideals as above is automatically satisfied for simple semigroups. 

For a positively ordered, strongly finite semigroup $M$, strong order-cancellation into ideals implies order-cancellation into ideals. Indeed, if $x+z\leq y+z$ with $x,y$ in some order ideal $I$, then $x\in I(y)$ by \autoref{lma:SameIdeal} and we have $I(y)\subseteq I$. But now there is $v\in I(x)\subseteq I$ such that $x+v\leq y+v$. Likewise, if $x+z=y+z$ for $z$ an element in an order ideal $I$, then by the above argument there are elements $v,w\in I$ such that $x+v\leq y+v$ and $y+w\leq x+w$, from which it follows that $x+(v+w)=y+(v+w)$. 
\end{pgr}

Cancellation into ideals is related to full cancellation via the notion of separativity, as we show below:

\begin{prp}
\label{prp:cvsci} Let $M$ be a positively ordered preminimal semigroup. Then the following conditions are equivalent:
\begin{enumerate}[{\rm (i)}]
\item $M$ is (order-)cancellative.
\item \begin{enumerate}[{\rm (a)}]
\item $M$ is (order-)separative.
\item $M$ has (order-)cancellation into ideals.
\item $M$ is strongly finite.
\end{enumerate}
\end{enumerate}

Moreover, $M$ is order-cancellative if, and only if, it is strongly finite, nearly separative and has strong order-cancellation into ideals.

\end{prp}
\begin{proof}
It is clear that (i) implies  (ii). 
Assume condition (ii) and let us show that (i) holds. Suppose that $x+z\leq y+z$. Let $I=I(y)$. By \autoref{lma:SameIdeal}, $I(x)\subseteq I$. By (b) we may assume that $z$ belongs to $I$, hence there is $k\in\mathbb{N}$ such that $z\leq ky$. Then order-separativity entails $x\leq y$, by \cite[Theorem 1.4]{WehComm}.

If $x+z=y+z$, then the above argument shows that $x\leq y$ and $y\leq x$, whence $x=y$ as $M$ is partially ordered.

Concerning the second equivalence, we only need to verify the `if' direction. Thus, assume suppose that $x+z\leq y+z$. Again by \autoref{lma:SameIdeal}, $x$ is an element of $I(y)$. By strong order-cancellation into ideals, we can find $v\in I(x)$ such that $x+v\leq y+v$. Finally, \autoref{lma:Wsep} implies that $x\leq y$. In case $x+z=y+z$, we argue as in the previous paragraph.
\end{proof}

\begin{rmk}
In the simple case, \autoref{prp:cvsci} states the well-known fact that order-cancellation is equivalent to order-separativity and finiteness.
\end{rmk}

\begin{pgr}[Grothendieck groups] Given a semigroup $M$, we shall denote by $\Gr(M)$ its Gro\-then\-dieck group that can be preordered by taking as positive cone $\Gr(M)^+=\iota(M)$, where $\iota\colon M\to \Gr(M)$ is the Grothendieck map. If $M$ is finite, then $\Gr(M)$ with the order just defined is a partially ordered abelian group.

We can also equip $\Gr(M)$ with a preorder defined by taking as positive cone
\[
\Gr(M)^{++}=\{\iota(x)-\iota(y)\mid y\leq x \text{ in }M\}\,.
\] 
It is easy to verify that $(\Gr(M), \Gr(M)^{++})$ is partially ordered whenever $M$ is. In general, $\Gr(M)^{+}\subseteq\Gr(M)^{++}$, and equality occurs if $M$ is algebraically ordered. (Conversely, if $M$ is cancellative, $\Gr(M)^+=\Gr(M)^{++}$ implies that the order is algebraic.)
\end{pgr}

The notion of cancellation into ideals for positively ordered semigroups can be expressed, in the language of groups, as follows:

\begin{lma}
\label{lma:cancintoideals}
Let $M$ be a positively ordered semigroup. Then
\begin{enumerate}[{\rm (i)}]
\item $M$ has cancellation into ideals if, and only if, for every (principal) order ideal $I$ of $M$, the natural map $\Gr(I)\to \Gr(M)$ is injective.

\item $M$ has order-cancellation into ideals if, and only if, for every (principal) order ideal $I$ of $M$, the natural map $\Gr(I)\to\Gr(M)$ is an order-embedding with respect to the orderings induced by $\Gr(I)^{++}$ and $\Gr(M)^{++}$.
\end{enumerate}
\end{lma}
\begin{proof}
The proof is a straightforward application of the definitions.
\end{proof}

This result allows us to translate the notion to \ca{}s, which gives a characterization of the condition in $\K$-theoretic terms via index maps of exact sequences (see below). We say that a \ca\ $A$ has \emph{cancellation into ideals} provided its Murray-von Neumann semigroup $\V(A)$ has the corresponding property. Of course, if $A$ has (stable) cancellation of projections (for example, if $A$ has stable rank one), then $A$ has cancellation into ideals, for in that case $\V(A)$ is a cancellative semigroup. 

Recall that, if an ideal $I$ contains a full projection $p$, then $\K_0(I)=\Gr(\V(I))$. Indeed,  we have
\[
\V(I)\cong \V(I\otimes \KK)\cong \V(pIp\otimes\KK)\cong\V(pIp)\,.
\]
Since $\K_0(J)=\Gr(\V(J))$ holds for any unital C*-algebra $J$, we get
\[
\Gr(\V(I))\cong\Gr(\V(pIp))\cong\K_0(pIp)\cong\K_0(I)\,.
\]

Also note that if $p$ is a projection in $A\otimes \KK$, then the order ideal $I([p])\cong \V((A\otimes K)p(A\otimes \KK))$.   

\begin{prp}\label{prop:CancIdeals}
Let $A$ be a unital \ca. Then
\begin{enumerate}[{\rm (i)}]
 \item 
$A$ has cancellation into ideals if and only if the index map $\delta_I\colon\K_1(A/I)\to \K_0(I)$ is zero for all ideals $I$ of $A$ containing a full projection.
\item 
$A$ has order-cancellation into ideals if and only if the natural map $i_*\colon\K_0(I)\to \K_0(A)$ is an order-embedding for all ideals $I$ of $A$ containing a full projection. 
\end{enumerate}
\end{prp}

\begin{proof}
(i): Using the six term exact sequence, the index map $\delta_I\colon\K_1(A/I)\to \K_0(I)$ is zero for an ideal $I$ of $A$ if and only if the inclusion map of $I$ into $A$ induces an injective map from $\K_0(I)$ into $\K_0(A)$. By \autoref{lma:cancintoideals}, $V(A)$ has cancellation into ideals if and only if for every principal order ideal $I([p])$ of $V(A)$, the natural map $\Gr(I([p]))\to \Gr(V(A))=\K_0(A)$ is injective (where $p$ is a projection in $A\otimes K$).

Set $I=(A\otimes \KK)p(A\otimes \KK)^{-}$ which clearly contains a full projection, and  now the if part of the statement follows from the fact that $I([p])\cong V(I)$. 
Conversely, if $\V(A)$ has cancellation into ideals, and $I$ is an ideal of $A$ containing a full projection $p$, the result follows by considering the order ideal $I([p])$ in $V(A)$.

(ii): This is similar to (i).
\end{proof}

\begin{prp}
\label{prp:cancidealsrr0}
Let $A$ be a \ca{} of real rank zero. If $A$ has cancellation into ideals, then so does $A\otimes\mathcal Z$.
\end{prp}
\begin{proof}
Let $\tilde{I}$ be an ideal of $A\otimes\mathcal Z$. Then there is an ideal $I$ of $A$ such that $\tilde{I}=I\otimes\mathcal Z$, by \cite[Lemma 2.13]{BlanKirJFA04}.

Since $A$ has real rank zero, we have $\K_0(I)=\Gr(\V(I))$ for any ideal $I$ of $A$ (see, for example, \cite{PerJot00}). We can now apply \autoref{lma:cancintoideals} to conclude that the map 
\[
\K_0(I)=\Gr(\V(I))\to\Gr(\V(A))=\K_0(A)
\]
is injective. As we have a commutative diagram:
\[
\xymatrix{
\K_0(I\otimes\mathcal Z)\ar[r] & \K_0(A\otimes\mathcal Z)   \\
\K_0(I)\ar[r]\ar[u]^{\cong}  & \K_0(A)\ar[u]^{\cong}
} 
\]
the conclusion follows.
\end{proof}

By \autoref{prop:CancIdeals}, we have that for a unital C*-algebra $A$, stable cancellation of projections implies that the index map $\delta_I$ is zero for all ideals $I$ containing a full projection. However, it is not true that the index map is zero for every ideal of $A$. For an example consider the algebra of continuous functions on the closed unit disk and the ideal consisting of those functions vanishing on the boundary. Hence, cancellation into ideals is both, strictly weaker than stable cancellation, and strictly weaker than the assumption that the index map is zero for all ideals.

It also follows from the argument in \autoref{prp:cancidealsrr0} that, if $A$ is a \ca{} of real rank zero, then $A$ has cancellation into ideals precisely when the index map $\delta_I$ vanishes for all ideals $I$ of $A$. Equivalently, when the natural map $\K_1(A)\to\K_1(A/I)$ is surjective for all ideals $I$.

Recall that a \ca{} $A$ is termed \emph{separative} (\cite{AraGoodearlOMearaPardo98}) if the Murray-von Neumann semigroup $\V(A)$ is separative. 
Separativity appears quite often in \ca{}-theory. For example, all known examples of \ca{s} with real rank zero are separative, and it remains an open problem to decide whether this is always the case.

Separativity in a \ca{} $A$ is known to be equivalent to the following condition, introduced in \cite{BroPedExtRich}, and termed \emph{weak cancellation} there: given projections $p$ and $q\in M_n(A)$ that generate the same closed two-sided ideal and such that $[p]=[q]$ in $\K_0(I)$ are necessarily Murray-von Neumann equivalent. Appealing to \cite[Theorem 3.5 and Remark 3.11]{Jiang97}, it follows that $A\otimes \ZZ$ is separative for any \ca{} $A$.

\begin{prp}
\label{prp:sep}
Let $A$ be a stably finite, $\ZZ$-stable \ca{} of real rank zero. Then $\Cu(A)$ is nearly separative.
\end{prp}
\begin{proof}
From the comments above we know that $\V(A)$ is separative. We also have $\Cu(A)$ algebraic as $A$ has real rank zero. Then apply \autoref{cor:sep}.
\end{proof}

\begin{prbl}
\label{pr:nearsep}
If $A$ is a $\ZZ$-stable \ca{}, is $\Cu(A)$ nearly separative? Is it weakly separative? (See \autoref{pr:nearunp}.)
\end{prbl}

According to well established terminology, we say that $A$ is \emph{residually stably finite} if all quotients of $A$ are stably finite. Since $\Cu(A/I)\cong\Cu(A)/\Cu(I)$ (see \cite{CiuRobSan11}), this condition is equivalent to the demand that $\Cu(A)$ is strongly finite. In the case of $\V(A)$, we have the following:

\begin{lma}
\label{lma:stfinite}
Let $A$ be a \ca. If $A$ is residually stably finite, then $\V(A)$ is strongly finite. The converse holds if $A$ has real rank zero.
\end{lma}
\begin{proof}
Since the property of being residually stably finite is a stable property, we may assume that $A$ is stable.

For an ideal $J$ of $A$, let $\pi_J\colon A\to A/J$ denote the natural quotient map. This induces a semigroup map 
\[
\rho_J\colon\V(A)/\V(J)\to\V(A/J)
\]
 by $\rho_J(\pi_{\V(J)}([p]))=[\pi_J(p)]$, where $\pi_{\V(J)}\colon\V(A)\to\V(A)/\V(J)$ is the natural map.

Next, let $I$ be an order ideal of $\V(A)$, and let 
\[
J_0:=\{p\in A\mid p\text{ is a projection and }[p]\in I\}\,.
\]
Then $J:=\overline{AJ_0A}$ is an ideal with $\V(J)=I$.

Suppose that $\overline{x}+\overline{y}=\overline{x}$ in $\V(A)/\V(J)$. Taking representatives $p$ and $q$ for $x$ and $y$ respectively, this means that $\pi_J(p)\oplus\pi_J(q)\sim\pi_J(p)$ in $A/J$. By assumption, this forces $\pi_J(q)=0$, and thus $q\in I$, so that $y=0$.
 
Now assume that $A$ has real rank zero. Then we know from \cite[Proposition 1.4]{AraGoodearlOMearaPardo98} that $\rho_J$ is an isomorphism for any ideal $J$ of $A$, and thus $A/J$ is residually stably finite if $\V(A)$ is strongly finite.
\end{proof}

\begin{prp}
\label{cor:cancV}
Let $A$ be a unital \ca{} of real rank zero that has cancellation into ideals (in particular, this holds if $A$ has real rank zero and stable rank one). Then
\[
A\otimes\mathcal Z \text{ is residually stably finite }\implies\V(A\otimes\mathcal Z) \text{ has cancellation.}
\]
\end{prp}
\begin{proof}
By \autoref{prp:cancidealsrr0} $\V(A\otimes\mathcal{Z})$ has cancellation into ideals. If we further assume that $A\otimes\mathcal{Z}$ is residually stably finite, then by \autoref{lma:stfinite} $\V(A\otimes\mathcal{Z})$ is strongly finite. Since by the comments previous to \autoref{prp:sep} $A\otimes \ZZ$ is always separative, the conclusion follows from \autoref{prp:cvsci}.
\end{proof}

\begin{rmk}
It appears plausible that, in general, $A\otimes\mathcal{Z}$ will be stably finite whenever $A$ is. This would improve \autoref{cor:cancV} to show that if $A$ has real rank zero and is residually stably finite, then $A\otimes\mathcal{Z}$ has stable cancellation of projections. Note that, if $A$ has stable rank one and $\K_1(I)=0$ for any ideal $I$ of $A$, then already $A\otimes\mathcal{Z}$ has stable rank one (\cite[Corollary 1.6]{AntoineBosaPereraPetzka14}). It is not known (to the authors) whether $A\otimes\mathcal{Z}$ will have stable rank one whenever $A$ has.
\end{rmk}

\section{Divisibility and perforation}
\label{sec:divisibility}

\begin{dfn}[Divisibility conditions]
\label{dfn:div}
Recall that a semigroup $M$ is \emph{almost divisible} if, for any $x\in M$ and $n\in\N$, there exists an element $z\in M$ such that $nz\leq x\leq (n+1)z$.

Recall also that $M$ is said to be \emph{weakly divisible} if for any $x\in M$ and $n\in\N$, there are elements $y$, $z\in M$ such that $x=ny+(n+1)z$.
\end{dfn}


\begin{lma}
\label{lem:almostdivvsweakdiv}
Let $M$ be an algebraically ordered semigroup. If $M$ is weakly divisible, then it is almost divisible. The converse holds if, moreover, $M$ is cancellative.
\end{lma}
\begin{proof}
Suppose that $M$ is weakly divisible. Given $x$ and $n$ there are $s,t\in M$ such that $x=ns+(n+1)t$. Then $n(s+t)\leq x\leq (n+1)(s+t)$. 

Conversely, suppose that $M$ is cancellative and almost divisible. Let $x\in M$, and $n\in \N$. There is $z\in M$ such that $nz\leq x\leq (n+1)z$. Find $s,t\in M$ such that $nz+t=x$ and $x+s=(n+1)z$. Then $x+z=(n+1)z+t=x+s+t$, and so by cancellation $z=s+t$.

Now $x+s=(n+1)z=(n+1)s+(n+1)t$, whence $x=ns+(n+1)t$, again by cancellation.
\end{proof}

\begin{pgr}[Perforation conditions]
\label{pgr:perf}
Given a semigroup $M$ and elements $x$, $y\in M$, write $x<_sy$ provided there is $n\in\N$ such that $(n+1)x\leq ny$.
Recall that a semigroup $M$ is termed \emph{almost unperforated} if $x<_s y$ implies $x\leq y$. 

As in \cite{APT14}, for a semigroup $M$ we write $x\leq_p y$ provided there is $n\in\N$ such that $nx\leq ny$ and $(n+1)x\leq (n+1)y$. We say that $M$ is \emph{nearly unperforated} if $x\leq_p y$ implies $x\leq y$. That near unperforation implies almost unperforation is easy to show (see \cite[Proposition 5.6.3]{APT14}). The converse holds under additional assumptions (see \cite{APT14} and the discussion below). It is also easy to verify that near unperforation implies near separativity as defined in \ref{sct:separat}.

Recall that a partially ordered abelian group $(G,G^+)$ is said to be almost unperforated if, whenever $nx$, $(n+1)x\in G^+$, it follows that $x\in G^+$. R\o rdam proved for an ordered group $(G,G^+)$ in \cite[Lemma 3.4]{Rordam04} that $(G,G^+)$ is almost unperforated if and only if $G^+$ is almost unperforated as a semigroup. Therefore, we conclude from our observation above that $(G,G^+)$ is almost unperforated if and only if $G^+$ is nearly unperforated.
\end{pgr}

We recall the following problem, posed in \cite{APT14}:
\begin{prbl}
\label{pr:nearunp} Let $A$ be a $\ZZ$-stable \ca{}. Is $\Cu(A)$ always nearly unperforated?
\end{prbl}
As shown by R\o rdam in \cite[Theorem 4.5]{Rordam04}, $\Cu(A)$ is almost unperforated if $A$ is $\mathcal{Z}$-stable. A positive answer to the above problem is known in a number of cases, for example, when $A$ has no $\K_1$-obstructions. (see \cite[5.6.19]{APT14}).

The next result implies that, for cancellative, algebraically ordered semigroups, almost unperforation and near unperforation are equivalent properties. Part of the argument is based on \cite[Proposition 5.6.12]{APT14}, and we include it for completeness.

\begin{prp}
\label{prp:equivnuau}
Let $M$ be a strongly finite, algebraically ordered semigroup. Then the following statements are equivalent.
\begin{enumerate}[{\rm (i)}]
\item $M$ is nearly unperforated and has cancellation into ideals.
\item $M$ is almost unperforated and cancellative.
\end{enumerate}
\end{prp}

\begin{proof}
(i)$\implies$(ii): As already mentioned in \ref{pgr:perf}, near unperforation implies almost unperforation. It was proved in \cite[Proposition 5.6.3]{APT14} that, if $M$ is nearly unperforated, then $M$ is nearly separative. Since the order is algebraic, we see that $M$ is actually separative.

Now, let $x,y,z$ be elements in $M$ such that $x+z=y+z$. By \autoref{lma:SameIdeal}, $I=I(x)=I(y)$, and thus we may assume that $z$ also belongs to $I$. By separativity, we obtain $x=y$.

(ii)$\implies$(i): Let $x,y$ be elements in $M$ and $n\in \N$ such that $nx\leq ny$ and $(n+1)x\leq (n+1)y$. As the order is algebraic, we can find $c,d\in M$ with $nx+c=ny$ and $(n+1)x+d=(n+1)y$. Hence $$(n+1)nx+(n+1)c=(n+1)ny=n(n+1)x+nd.$$ Cancellation gives us $(n+1)c=nd$ and almost unperforation shows that $c\leq d$. We find $e\in M$ with $c+e=d$. Now,
$$c+e+(n+1)x+ny=d+(n+1)x+ny=(n+1)y+nx+c.$$
Cancelling elements we get $e+x=y$, hence $x\leq y$. 
\end{proof}

\begin{cor}
Let  $M$ be a strongly finite, algebraically ordered semigroup such that $M$ has cancellation into ideals. Then, $M$ is nearly unperforated if and only if $M$ is separative and almost unperforated.
\end{cor}

\begin{proof}
Under the assumptions of the corollary all notions of separativity discussed in Section \ref{sct:separat} agree and, by \autoref{prp:cvsci}, separativity is equivalent to cancellation. The equivalence then follows from \autoref{prp:equivnuau}.
\end{proof}

For simple $\CatCu$-semigroups, we have the following
\begin{prp}
\label{prp:NsepCu} Let $S$ be a simple, stably finite $\CatCu$-semigroup that satisfies \axiomO{5} and is almost unperforated. Then $S$ is nearly separative if, and only if, $S_c$ is separative.
\end{prp}
\begin{proof}
First note that since $S$ is stably finite and satisfies \axiomO{5}, $S_c$ is nearly separative precisely when it is separative (using, for example, \autoref{lma:Wsep}).

It is clear that if $S$ is nearly separative, so is $S_c$. For the converse, we can use \cite[Proposition 5.3.16]{APT14} to conclude that every element is either compact or non-zero and soft (in the sense of \cite[Definition 5.3.1]{APT14}). Now assume $2x\leq x+y\leq 2y$ in $S$. If either $x$ or $y$ are soft, then the argument in \cite[Theorem 5.6.10]{APT14} applies to conclude $x\leq y$. Hence we may assume $x$, $y\in S_c$, in which case near separativity of $S_c$ implies $x\leq y$.
\end{proof}

We say that a $\CatCu$-semigroup $S$ has \emph{weak cancellation into ideals} if $x+z\ll y+z$ with $x,y$ contained in some ideal $I$, then $x+w\ll y+w$ for some $w\in I$. Observe that this condition is always satisfied in the simple case.

\begin{thm}
\label{thm:o5nunpaunp}
Let $S$ be an algebraic $\CatCu$-semigroup that satisfies \axiomO{5}. Then the following conditions are equivalent:
\begin{enumerate}[{\rm (i)}]
\item $S$ is strongly finite, nearly unperforated, and has weak cancellation into ideals.
\item $S$ is almost unperforated and weakly cancellative.
\end{enumerate}
\end{thm}
\begin{proof}
(ii)$\implies$(i) can be easily derived using \cite[Proposition 5.6.12]{APT14}, or also Pro\-po\-si\-tion \ref{prp:equivnuau}.

(i)$\implies$(ii): We know that $S$ is nearly separative, hence $S_c$ is a separative semigroup by \autoref{cor:sep}.

Assume $x+z\ll y+z$. Since $S$ is strongly finite, we see that $x\in I(y)$. Using weak cancellation into ideals, we may assume that $z\in I(y)$. Using that $S$ is algebraic, we may reduce to the situation where $x$, $y$ and $z$ all belong to $S_c$ (see, for example, the arguments in \autoref{lma:preminimal}), with $z\in I(y)$. Find $t\in S_c$ and $n\in\N$ such that $x+ny+t=ny+y$. Another use of strong finiteness yields $y\in I(x+t)$. Finally, separativity implies $x+z=y$.
\end{proof}

\begin{prbl}
Does \autoref{thm:o5nunpaunp} remain valid without the assumption of $S$ being algebraic?
\end{prbl}

\section{Constructing almost unperforated cones}
\label{sec:cones}

In this section we study the problem of mapping a positively ordered semigroup into an \emph{algebraically ordered} and \emph{almost unperforated} semigroup in a universal way. Wehrung showed in \cite[Proposition 4.3]{WehCan} that every positively ordered semigroup which is strongly preminimal can be embedded in an algebraically ordered semigroup. We will not define strong preminimality here, but just mention that it is a weakening of separativity. Thus, in particular, every separative semigroup can be embedded in an algebraically ordered semigroup. The difference in our approach resides in that we seek to obtain an additional unperforation condition.

As we shall see, a natural candidate for this is (the positive cone of) the Grothendieck group of the semigroup, though this will also impose cancellation in the resulting semigroup. We first analyse the setting when the given semigroup is already almost unperforated.

\begin{lma}
\label{lma:almunp}
Let $M$ be a strongly finite, algebraically ordered semigroup with cancellation into ideals. If $M$ is nearly unperforated or almost unperforated, then $\Gr(M)^+$ is almost unperforated. 
\end{lma}

\begin{proof}
Since near unperforation implies almost unperforation, it is enough to assume $M$ is almost unperforated.

Recall that $\iota$ denotes the Grothendieck map $\iota\colon M\to \Gr(M)$ and $\Gr(M)^+=\iota(M)$. We have to show that $\iota(M)$ is almost unperforated. Suppose there is $n\in\N$ such that $(n+1)\iota(x)\leq n\iota(y)$. Then, as $M$ is algebraically ordered, there are $z,z'$ in $M$ such that $(n+1)x+z+z'= ny+z$ in $M$. By \autoref{lma:SameIdeal} one gets that $I:=I((n+1)x+z')=I(y)$. By cancellation into ideals we can now exchange $z$ for a $\tilde{z}\in I$ such that $(n+1)x+\tilde{z}+z'= ny+\tilde{z}$. As in \cite[Theorem 2.1.9]{BlackadarRationalC-algebras}, we find then some $L\in\N$ such that $L(n+1)x+Lz'=Lny$. In particular $(Ln+1)x\leq Ln y$ in $M$. Since $M$ is almost unperforated, $x\leq y$. It follows that $\iota(x)\leq \iota(y)$, which shows almost unperforation of $\iota(M)$, as desired.
\end{proof}

Combining this with \autoref{lma:stfinite} we immediately obtain:

\begin{cor}\label{cor:almostUnperforation}
\label{cor:almunperfv}
Let $A$ be a residually stably finite \ca{} that has cancellation into ideals. Then, if $\V(A)$ is almost unperforated so is $\K_0(A)^+$.
\end{cor}

\begin{prp}
\label{lma:Gau}
Let $(G,G^+)$ be a partially ordered abelian group. Put
\[
G_{\mathrm{au}}^+=\{x\in G\mid nx,\, (n+1)x\in G^+\text{ for some }n \}\,.
\]
Then $G_{\mathrm{au}}^+$ is a strict cone and $(G,G_{\mathrm{au}}^+)$ is almost unperforated. Therefore, $(G,G^+)$ is almost unperforated if and only if $G^+=G_{\mathrm{au}}^+$.
\end{prp}
\begin{proof}
Let $x,y\in G_{\mathrm{au}}^+$. Then, there are $k$, $l$ such that $kx, (k+1)x\in G^+$ and $ly, (l+1)y\in G^+$. Put $N=3kl+k+l$. Then $Nx=klx+l(k+1)x+(l+1)kx\in G^+$, and likewise $Ny\in G^+$. Thus $N(x+y)\in G^+$. Also, $N+1=2kl+(k+1)(l+1)$, so that both $(N+1)x$ and $(N+1)y$ belong to $G^+$. It follows that $(N+1)(x+y)\in G^+$ and thus $G_{\mathrm{au}}^+$ is a cone.

Suppose now that $x\in G$ satisfies that $\pm x\in G_{\mathrm{au}}^+$. Then, there are $k$ and $l\in\N$ such that $kx, (k+1)x\in G^+$ and also $-lx, -(l+1)x\in G^+$. Then, since $\pm klx\in G^+$, we see that $klx=0$. Arguing similarly, we obtain that $k(l+1)x=(k+1)lx=0$, so that $kx=lx=0$. Since also $(k+1)(l+1)x=0$, it follows that $x=0$.

Finally, suppose that $nx, (n+1)x\in G_{\mathrm{au}}^+$. We are to show that $x\in G_{\mathrm{au}}^+$. By definition, we can find $k$, and $l\in \N$ such that $knx, (k+1)nx\in G^+$, and $l(n+1)x, (l+1)(n+1)x\in G^+$. Put $L=2kn+n+l+ln$. Then $Lx=knx+(k+1)nx+l(n+1)x\in G^+$, and $(L+1)x=2knx+(l+1)(n+1)x\in G^+$, as desired.

The last part of the statement is clear.
\end{proof}

\begin{rmk}
\label{rm:VZ}
Suppose that $G$ is simple as a partially ordered abelian group (that is, every non-zero, positive element is an order-unit). Then 
\[
G_{\mathrm{au}}^+=\{x\in G\mid nx\in G^+\setminus\{0\} \text{ for some }n\}\cup\{0\}\,.
\]
Indeed, if $nx>0$, then for $N$ large enough $Nnx\geq x$, and then $(Nn-1)x\in G^+$, and also $Nnx\in G^+$. Conversely, if $nx$, $(n+1)x\in G^+$ and $x\neq 0$, then $nx\neq 0$ since otherwise $x>0$, and so $kx>0$ for any $k$.
\end{rmk}

\begin{pgr}
\label{pgr:au}
Let $M$ be a finite, positively ordered semigroup. We shall denote $\Gr(M)^+_{\mathrm{au}}$ to refer to the construction in \autoref{lma:Gau} with respect to the partially ordered abelian group $(\Gr(M),\Gr(M)^+)$. 

Since $\Gr(M)^+\subseteq\Gr(M)_{\mathrm{au}}^+$, we have a natural map $\bar{\iota}\colon M\to\Gr(M)_{\mathrm{au}}^+$, defined by composing the Grothendieck map $\iota$ with the inclusion. This map is order-preserving when $M$ is algebraically ordered. Note that, if $M$ is cancellative, then $M$ is almost unperforated precisely when $M\cong\Mau$ (via $\bar{\iota}$).

The cone $\Gr(M)_{\mathrm{au}}^+$ defines a preorder $\leq_{\mathrm{au}}$, which is in fact an order by \autoref{lma:Gau}, as follows: $x\leq_{\mathrm{au}} y$ if and only if $y-x\in \Gr(M)_{\mathrm{au}}^+$. 

Observe that, if $N$ is a strict cone in $\Gr(M)$ such that $\iota(M)\subseteq N$ and $(\Gr(M),N)$ is almost unperforated, then $\Mau\subseteq N$. Indeed, if $x\in\Mau$, then $nx$, $(n+1)x\in\iota(M)\subseteq N$ for some $n$, and then $x\in N$. Note also that one can have strict cones $N$ that contain $\Mau$, yet $(\Gr(M),N)$ is not almost unperforated. For example, take $M=\N\times\N$, so that $\Gr(M)=\Z\times\Z$, and let $N=M+(3,-3)\N$. In this case, $(2,-1)\notin N$, but $2(2,-1)$, $3(2,-1)\in N$, and so $(\Z\times\Z,N)$ is not almost unperforated.

We may also look at the cone
\[
\Gr(M)^{++}=\{\iota(x)-\iota(y)\mid y\leq x\text{ in }M\}\,,
\]
as defined in \ref{pgr:cancintoideals}, which is a strict cone. The natural map $\iota^+\colon M\to \Gr(M)^{++}$ is order-preserving, and it is an order-embedding if and only if $M$ is order-cancellative, as is easy to verify. In this situation, we shall denote by $\bar{\iota}^+\colon M\to\Gr(M)^{++}_\mathrm{au}$ the composition of $\iota^+$ with the inclusion $\Gr(M)^{++}\subseteq\Gr(M)^{++}_\mathrm{au}$, which is an order-preserving semigroup morphism.

Since we can construct the cones $\Mau$ and $\Gr(M)^{++}_{\mathrm{au}}$, both sitting between $\Gr(M)^{+}$ and $\Gr(M)$, it is natural to look at their relative position. By the comments above, as $(\Gr(M),\Gr(M)^{++}_{\mathrm{au}})$ is almost unperforated (by \autoref{lma:Gau}), it follows that $\Mau\subseteq\Gr(M)^{++}_{\mathrm{au}}$ in general.

To construct an example where equality does not occur, let $\theta$ be an irrational number, and let $M$ be the subsemigroup of $\mathbb{R}^+$ generated by positive integer combinations of $1$ and $\theta$. Equip $M$ with the natural order, which is clearly not algebraic.

As a semigroup, $M\cong\N\times\N$. Note that $\Gr(M)\cong\Z\times\Z$. With this identification, we have
\[
\Gr(M)^+=\{(a,b)-(c,d)\mid a\geq c, b\geq d\}=\N\times\N\,,
\]
whereas
\[
\Gr(M)^{++}=\{(a,b)-(c,d)\mid c+d\theta\leq a+b\theta\}\,.
\]
Observe that $\Mau=\N\times\N$ and $\Gr(M)^{++}\not\subseteq\Mau$. This already implies that $\Mau\neq\Gr(M)^{++}_\mathrm{au}$, although in this case we have $\Gr(M)^{++}_\mathrm{au}=\Gr(M)^{++}$, as is easy to verify.
\end{pgr}

\begin{prp}
\label{prp:Gau} 
Let $M$ be a positively ordered semigroup. Then:
\begin{enumerate}[{\rm (i)}]
\item If $M$ is nearly unperforated and order-cancellative, then $\Gr(M)^{++}$ is nearly unperforated.
\item The map $\bar{\iota}^+\colon M\to\Gr(M)^{++}_\mathrm{au}$ is an order-embedding if and only if $M$ is order-cancellative and nearly unperforated.
\item If $M$ is algebraically ordered, then $\bar{\iota}\colon M\to \Gr(M)^+_\mathrm{au}$ is an order-embedding if and only if $M$ is cancellative and almost unperforated.
\end{enumerate}
\end{prp}

\begin{proof}
(i): Suppose that $M$ is nearly unperforated and order-cancellative. Let $x,y,u,v$ be elements in $M$ such that $y\leq x$, $v\leq u$ and $n(\iota^+(x)-\iota^+(y))\leq n(\iota^+(u)-\iota^+(v))$ for $n\geq n_0\in \mathbb{N}$. Then, using order-cancellation on $M$ we obtain $nx+nv\leq  ny+nu$ for $n\geq n_0$. Near unperforation implies that $x+v\leq y+u$, and thus $\iota^+(x)-\iota^+(y)\leq \iota^+(u)-\iota^+(v)$ in $\Gr(M)^{++}$.

(ii): We only have to prove the `if' direction. Thus assume that $M$ is order-cancellative and nearly unperforated. By (i), we know that $\Gr(M)^{++}$ is nearly unperforated, and thus also almost unperforated. Applying \autoref{lma:Gau}, we conclude that $\Gr(M)^{++}=\Gr(M)^{++}_\mathrm{au}$. Finally, since $M$ is order-cancellative, $\bar{\iota}^+=\iota^+=$ is an order-embedding, as desired.

(iii): This is just a reformulation of (ii) in the algebraically ordered setting.
\end{proof}
\begin{rmks}

\begin{enumerate}[{\rm (a)}]
\item That near unperforation in $M$ passes to $\Gr(M)^{++}$ as in Pro\-po\-si\-tion \ref{prp:Gau} remains true if one starts with a preminimally, positively ordered semigroup $M$ with strong order-cancellation into ideals. 
\item If $M$ is simple, in the sense that every non-zero element is an order-unit, then condition (ii) in \autoref{prp:Gau} can be stated as follows: $\iota(x)<_{\mathrm{au}}\iota(y)$ if and only if $(k+1)x\leq ky$ for some $k$.
\end{enumerate}

\end{rmks}

\begin{prp}
\label{prp:VZweaklydiv}
Let $M$ be a finite, algebraically ordered semigroup. If $M$ is weakly divisible, then $\Mau$ is also weakly divisible.
\end{prp}
\begin{proof}
Let $x\in \Mau$ and $n\in \N$. By definition, there is $k\in\N$ such that $kx=\iota(e')$, with $e'\in M$. Choose $m\in\N$ such that $2n<mk-1$. Let $e=me'$, and then $mkx=\iota(e)$.

Notice that our choice of $m$ above ensures that $(mk+1)n<(mk-1)(n+1)$. Put $L=(mk+1)n$. As $M$ is almost divisible, there is $z\in M$ such that $Lz\leq e\leq (L+1)z\leq (mk-1)(n+1)z$ in $M$. By applying $\bar\iota\colon M\to \Mau$, and taking into account that $\iota(e)=mkx$, we see that 
\[
(mk+1)n\bar\iota(z)\leq mkx\leq (mk-1)(n+1)\bar\iota(z)\,.
\]
Now, as $\Mau$ is almost unperforated, we obtain that $n\iota(z)\leq x\leq (n+1)\iota(z)$. This shows that $\Mau$ is almost divisible and, since it is cancellative, we obtain from \autoref{lem:almostdivvsweakdiv} that $\Mau$ is also weakly divisible.
\end{proof}


The assignment $M\mapsto\Gr(M)^{++}_\mathrm{au}$ is functorial and enjoys a universal property that we now establish.

\begin{thm}
\label{prp:universalau}
Let $M$ be a positively ordered semigroup, and let $N$ be an algebraically ordered, cancellative semigroup which is almost unperforated. Then, for any order-preserving semigroup map $f\colon M\to N$ there is a unique semigroup map $\bar{f}\colon\Gr(M)^{++}_\mathrm{au}\to N$ such that $\bar{f}\circ\bar{\iota}^+=f$.
\end{thm}
\begin{proof}
Identify $N$ with $\Gr(N)^+$. By functoriality of the Grothendieck construction, there is a group homomorphism $f_0\colon\Gr(M)\to\Gr(N)$ such that $f_0\circ \iota^+=f$. 

Define $\bar{f}=f_0\mid_{\Gr(M)^{++}_\mathrm{au}}$. We need to show that its image is a subset of $N$.

Given $x\in\Gr(M)^{++}_\mathrm{au}$, let $n\in\N$ be such that $nx$ and $(n+1)x$ belong to $\Gr(M)^{++}$. Choose $u,v,w,z$ in $M$ with $v\leq u$ and $z\leq w$ in $M$, and such that 
\[
nx=\iota^+(u)-\iota^+(v)\text{ and }(n+1)x=\iota^+(w)-\iota^+(z)\,.
\]
Rewrite the above equalities as $nx+\iota^+(v)=\iota^+(u)$ and $(n+1)x+\iota^+(z)=\iota^+(w)$. Now apply $\bar{f}$ to these equalities to obtain
\[
n\bar{f}(x)+f(v)=f(u)\text{ and } (n+1)\bar{f}(x)+f(z)=f(w)\,,
\]
which clearly implies $\bar{f}(x)+f(z)+f(u)=f(w)+f(v)$ in $\Gr(N)$. As $v\leq u$ and $z\leq w$ in $M$, after applying $f$ there are elements $a$ and $b\in N$ with $f(v)+a=f(u)$ and $f(z)+b=f(w)$. Using this and that $\Gr(N)$ is cancellative, we obtain
\[
\bar{f}(x)+a=b\,.
\]

Now, we also have $(n+1)nx+(n+1)\iota^+(v)=(n+1)\iota^+(u)$ and $n(n+1)x+\iota^+(z)=n\iota^+(w)$. Applying $\bar{f}$ to these equalities and using cancellation in $\Gr(N)$ again we get $(n+1)f(u)+nf(z)=(n+1)f(v)+nf(w)$ in $N$. Combine this equality with $f(v)+a=f(u)$ and $f(z)+b=f(w)$ to conclude, using cancellation in $N$, that $(n+1)a=nb$ in $N$. Since $N$ is almost unperforated, we get $a\leq b$, that is, $a+c=b$.

Finally, $\bar{f}(x)+a=b=a+c$, whence $\bar{f}(x)=c$ is an element in $N$, as desired.

The argument we have used also proves that $\bar{f}$ is unique.
\end{proof}


\begin{cor}
\label{cor:universalau}
 Let $M$ and $N$ be algebraically ordered semigroups, with $N$ cancellative and almost unperforated. Then, given a semigroup map $f\colon M\to N$, there is a unique semigroup map $\bar{f}\colon\Mau\to N$ such that $\bar{f}\circ\bar\iota=f$. 
\end{cor}
\begin{proof}
As observed in \ref{pgr:au}, in this case we have $\Gr(M)^+=\Gr(M)^{++}$, and thus also $\Mau=\Gr(M)^{++}_\mathrm{au}$ and $\bar{\iota}=\bar{\iota}^+$. The result follows from \autoref{prp:universalau}.
\end{proof}

\begin{cor}
\label{cor:functorialau}
Let $M$ and $N$ be positively ordered semigroups. Then, any semigroup map $f\colon M\to N$ induces a map $\bar{f}\colon\Gr(M)^{++}_\mathrm{au}\to\Gr(N)^{++}_\mathrm{au}$ such that $\bar{f}\circ\bar{\iota}_M^+=\bar{\iota}_N^+\circ f$.
\end{cor}
\begin{proof}
Apply \autoref{prp:universalau} with $\Gr(N)^{++}_\mathrm{au}$ in place of $N$ and $\bar{\iota}_N^+\circ f$ in place of $f$.
\end{proof}

We will now focus on how the order changes when passing from a \ca{} $A$ to its $\ZZ$-stabilisation $A\otimes\ZZ$, via the natural map $j\colon A\to A\otimes\ZZ$.

\begin{lma}
\label{lma:commdiag} Let $A$ be a \ca{} containing a full projection. Then there is a commutative diagram
\[
\xymatrix{
\V(A)\ar[r]^{\V(j)}\ar[d]^{\bar{\iota}} & \V(A\otimes\ZZ)\ar[d]^{\iota}   \\
\Gr(\V(A))^+_\mathrm{au}  & \K_0(A\otimes\ZZ)^+\ar[l]_{\K_0(j)^{-1}}
} 
\]
\end{lma}
\begin{proof}
We of course have a commutative diagram:
\[
\xymatrix{
\V(A)\ar[r]^{\V(j)}\ar[d] & \V(A\otimes\ZZ)\ar[d]   \\
\K_0(A)\ar[r]^{\K_0(j)}  & \K_0(A\otimes\ZZ)
} 
\]
in which the bottom row is an isomorphism and the downward arrows are the obvious ones. Since $A$ has a full projection, $\K_0(A)=\Gr(\V(A))$ and we have $\iota(\V(A))=\K_0(A)^+$. Let $x$ be an element in $\K_0(A)$ and suppose that $\K_0(j)(x)\geq 0$ in $\K_0(A\otimes\ZZ)$. Then, there are relatively prime numbers $m$, $n$ such that $\K_0(j_{m,n})(x)\geq 0$ in $K_0(A\otimes Z_{m,n})$. Arguing exactly as in the proof of \cite[Theorem 1.4]{GongJiangSu00}, we see that $mx\geq 0$ and $nx\geq 0$ in $\K_0(A)$.

Now, find $a$, $b\in\Z$ such that $am+bn=1$. Writing $a=a_1-a_2$, $b=b_1-b_2$ for $a_i$, $b_i$ positive integers, we see that, with $L=a_2m+b_2n\in \N$, we have that $L+1=a_1m+b_1n$. Thus $Lx$ and $(L+1)x$ both belong to $\K_0(A)^+$, and this means that $x\in \mathrm{Gr}(V(A))^+_{\mathrm{au}}$.

To complete the argument, take $y\in\K_0(A\otimes\ZZ)^+$. Since $\K_0(j)$ is an isomorphism, there is $x\in\K_0(A)$ with $\K_0(j)(x)=y$, and now the previous paragraph applies to show that in fact $x$ belongs to $\Gr(\V(A))^+_\mathrm{au}$.
\end{proof}

We now extend \cite[Theorem 1]{GongJiangSu00} to compute precisely what the positive cone of $\K_0(A\otimes\ZZ)$ is in terms of $\V(A)$.

\begin{thm}
\label{thm:GJSimproved}
Let $A$ be a \ca{} containing a full projection. Suppose that $A$ is residually stably finite and has cancellation into ideals. Let $j\colon A\to A\otimes\ZZ$ be the natural embedding. Then
\begin{enumerate}[{\rm (i)}]
\item $\K_0(A\otimes\ZZ)^+= j_*(\mathrm{Gr}(V(A))^+_{\mathrm{au}})$ and thus is partially ordered and almost unperforated.
\item $j_*\colon \K_0(A)\to \K_0(A\otimes\ZZ)$ is an isomorphism of ordered groups if and only if $\K_0(A)$ is almost unperforated.
\end{enumerate}
\end{thm}

\begin{proof}
(i). Take $x\in \mathrm{Gr}(V(A))^+_{\mathrm{au}}$. By \autoref{lma:commdiag}, we have to show that $j_*(x)\geq 0$ in $\K_0(A\otimes\ZZ)$. There is $n\in\N$ such that $nx$, $(n+1)x\in \K_0(A)^+$. It follows that $nj_*(x)$, $(n+1)j_*(x) \in K_0(A\otimes \mathcal{Z})^+$. But $K_0(A\otimes \mathcal{Z})^+=Gr(V(A\otimes \mathcal{Z}))^+$ is almost unperforated by \cite[Corollary 4.8]{Rordam04} and Corollary \ref{cor:almostUnperforation}, hence it is nearly unperforated as a cancellative almost unperforated semigroup and $j_*(x)\geq 0$.

(ii). As $j_*$ is already an isomorphism of abelian groups, we have that $j_*$ is an isomorphism of ordered groups precisely when $j_*(\K_0(A)^+)=\K_0(A\otimes\ZZ)^+$. By (i), this will happen if and only if $j_*(\K_0(A)^+)=j_*(\mathrm{Gr}(\V(A))^+_{\mathrm{au}})$, that is, $\K_0(A)^+=\mathrm{Gr}(V(A))^+_{\mathrm{au}}$. But this is equivalent to saying that $\K_0(A)$ is almost unperforated, by \autoref{lma:Gau}.
\end{proof}

\begin{rmk}
Note that the condition assumed on \autoref{thm:GJSimproved} is not necessary. An example to show this is the Toeplitz algebra $\mathcal T=C^*(s)$, where $s$ is a non-unitary isometry with $1-ss^*$ a rank one projection. It is well-known that $\mathcal T$ can be equivalently defined through a short exact sequence $0\rightarrow \KK \rightarrow \mathcal T \rightarrow C(\mathbb T)\rightarrow 0$. We know that $K_0(\mathcal T)\cong\mathbb Z$, with positive cone isomorphic to $\mathbb{N}_0$. On the other hand, $\K_0(\mathcal T\otimes\ZZ)=\Z$ and, since $C(\mathbb T, \ZZ)$ has stable rank one, we have that $\K_0(C(\mathbb T,\ZZ))^+=\V(C(\mathbb T,\ZZ))\cong\mathbb{N}_0$. It follows that $\K_0(\mathcal T\otimes\ZZ)^+ \cong\mathbb{N}_0 \cong  j_*(\mathrm{Gr}(\V(\mathcal T))^+_{\mathrm{au}})$ (but $\mathcal{T}$ is neither stably finite, nor does $\V(\mathcal{T})$ have cancellation into ideals).
\end{rmk}

\begin{cor}
 \label{cor:wkdivisible} Let $A$ be a \ca{} of real rank zero that contains a full projection. Suppose that $A$ is residually stably finite and has cancellation into ideals (in particular, this holds if $A$ has stable rank one). If $A$ is weakly divisible, so is $A\otimes\ZZ$.
\end{cor}
\begin{proof}
This is a consequence of \autoref{thm:GJSimproved} and \autoref{prp:VZweaklydiv}. 
\end{proof}

\begin{cor}
\label{cor:isthatyou?}
Let $A$ be a \ca{} that contains a full projection. Suppose that $A$ has real rank zero, cancellation into ideals, and $A\otimes\mathcal Z$ is residually stably finite.

Then the map $\V(j)\colon\V(A)\to\V(A\otimes\ZZ)$ is an order-embedding  if and only if $\V(A)$ is almost unperforated and cancellative, if and only if $\V(A)$ is nearly unperforated.
\end{cor}
\begin{proof}
By \autoref{cor:cancV}, $\V(A\otimes\mathcal{Z})$ is a cancellative semigroup, and thus it can be identified with $\K_0(A\otimes\mathcal{Z})^+$. We know by \cite[Corollary 4.8]{Rordam04} that $\V(A\otimes\ZZ)$ is always almost unperforated. 
Thus $\V(A)$ is almost unperforated and cancellative if $\V(j)$ is an order-embedding. Conversely, if $\V(A)$ is cancellative we may apply \autoref{thm:GJSimproved} to get $\K_0(A\otimes\ZZ)^+\cong\Gr(\V(A))^+_\mathrm{au}$, and $\V(j)$ is an order-embedding if, further, $\V(A)$ is almost unperforated, by \autoref{prp:Gau}.

The last equivalence follows from \autoref{prp:equivnuau}.
\end{proof}

Recall from \cite{AntoineBosaPereraPetzka14} (see also \cite{AntoinePereraSantiago11}) that a \ca{} $A$ has no $\K_1$-obstructions provided $A$ has stable rank one and $\K_1(I)=0$ for every ideal $I$ of $A$.

\begin{cor}
 Let $A$ be a \ca{} that contains a full projection. If $A$ has no $\K_1$-obstructions, then the map $\V(j)\colon\V(A)\to\V(A\otimes\ZZ)$ is an order-embedding if and only if $\V(A)$ is almost unperforated, if and only if $\V(A)$ is nearly unperforated.
\end{cor}
\begin{proof}
We know that both $A$ and $A\otimes\ZZ$ have stable rank one (the latter by \cite[Corollary 1.6]{AntoineBosaPereraPetzka14}), and thus $\V(A)$ and $\V(A\otimes\ZZ)$ are cancellative semigroups. Thus we may identify the latter with $\K_0(A\otimes\ZZ)^+$, which equals $\Gr(\V(A))^+_\mathrm{au}$, by \autoref{thm:GJSimproved}. Therefore, \autoref{prp:Gau} (iii) implies that $\V(j)$ is an order-embedding exactly when $\V(A)$ is almost unperforated. This, in turn, is equivalent to $\V(A)$ being nearly unperforated by \autoref{prp:equivnuau}.
\end{proof}


\section{Order in tensor products}
\label{sec:OrderTensor}

In this section we analyse the map $\V(A)\to\V(A)\otimes 1_Z$ and prove that it encodes the property of near unperforation in $\V(A)$. In the sequel, this result will be transported to the more general setting of $\CatCu$-semigroups. As a consequence, and as mentioned in the introduction, we will obtain partial answers to \cite[Problems 7.3.13 and 7.3.14]{APT14}.

Much of our discussion below will be carried out for a positively ordered semigroup $M$. Recall that $Z:=\Cu(\ZZ)$ is identified with $\N_0\sqcup (0,\infty]$ and that $1_Z=[1_\ZZ]$ corresponds to the compact element $1\in\N_0$. 

\begin{pgr}[Tensor products of positively ordered semigroups] 
\label{pgr:tensprod}
We briefly sketch the construction of the tensor product of positively ordered semigroups, since the concrete form of the partial ordering that can be given to them will be needed below (see also \cite{Weh96},\cite[Appendix B]{APT14}).
 
Given $M$ and $N$ positively ordered semigroups, one first constructs the tensor product as conical semigroups. Set $F=\N[M^\times \times N^\times]$ to be the free abelian semigroup with basis $M^\times \times N^\times$ and whose elements we denote by $a\odot b$, $a\in M^\times$, $b\in N^\times$. Given $f_0,g_0\in F$ we set $f_0\rightarrow^0 g_0$ if $f_0=a\odot b$ and $g_0=\sum_{i,j} a_i\odot b_j$ where $a=\sum_i a_i$ and $b=\sum_j b_j$. Then, given $f,g\in F$ we set $f\rightarrow g$ if $f=\sum_k f_k$, $g=\sum_k g_k$ and $f_k\rightarrow^0 g_k$. Setting $\cong$ as the congruence relation on $F$ generated by $\rightarrow$ and $\leftarrow$, one obtains $M\otimes N=F/_{\cong}$ to be the tensor product of $M$ and $N$ as conical semigroups. 

To obtain a partially ordered semigroup, one further considers the relations $f_0\leq^0 g_0$ whenever $f_0=0$ or $f_0=a'\odot b'$, $g=a\odot b$ and $a'\leq a$, $b'\leq b$. Then set $f\leq' g$ whenever $f=\sum_k f_k$, $g=\sum_k g_k$ and $f_i\leq^0 g_i$. This defines a preorder relation $\leq$ on $F$ as the additive transitive relation generated by $\cong$ and $\leq'$. Thus $f\leq g$ if there exists $n\in \N$ and $f_k,f_k'\in F$ for $k=0,\dots, n$ such that 
\[ f=f_0\leq' f_0'\cong f_1\leq' f_1'\cong\dots\cong f_n\leq f_n'=g.\]

Finally, antisymmetrizing this relation $\leq$ we obtain the tensor product $M\otimes N$ in the category of positively ordered semigroups. Let $\omega\colon F\to M\otimes N$ be the quotient map, denote the image of the generators by $a\otimes b$ and for general elements $f=\sum_i a_i\odot b_i$ write $\bar f:=\omega(f)=\sum_i a_i\otimes b_i$.
 
\end{pgr}

Given a positively ordered semigroup $M$, we denote by $M\otimes 1_Z$ or $M\otimes 1$ the subsemigroup of $M\otimes Z$ of elements that can be written as $x\otimes 1$ for some $x\in M$. We equip $M\otimes 1$ 
with the induced order from $M\otimes Z$ and we want to characterize when $x\otimes 1\leq y\otimes 1$ in terms of the order in $M$. Recall that, for elements $x$, $y$ in $M$, we write $x<_sy$ when there is $n\in\N$ such that $(n+1)x\leq ny$, and $x\leq_s y$ provided $x\leq y$ or else $x<_sy$. 

\begin{lma}
\label{lma:orderStensorZ}
Let $M$ be a positively ordered semigroup, and let $x$, $y\in M$. 
\begin{enumerate}
\item If $x\leq_s y$, then $x\otimes 1\leq y\otimes 1$ in $M\otimes 1$.
\item If $x\otimes 1\leq y\otimes 1$ then $x\leq_p y$ in $M$.
\end{enumerate}

\end{lma}
\begin{proof}

Statement (i) can be derived from more general results in \cite{APT14}, but it is nevertheless simple enough to give a self contained argument. First, we may clearly assume that $x,y\neq 0$ and $x<_s y$. Suppose $(n+1)x\leq ny$ for some $n\geq 1$ and then
\[
x\otimes 1\leq x\otimes 1+x\otimes\frac{1}{n}=(n+1)x\otimes\frac{1}{n}\leq ny\otimes\frac{1}{n}= y\otimes 1'\leq y\otimes 1\,.
\]

We now turn to (ii). Assume $x\otimes 1\leq y\otimes 1$. By the construction of the tensor product, there exists $m\in \N$ and $f_k,f_k'\in \N[M^\times \times Z^\times]$ for $k=0,\dots, m$ such that 
\begin{equation}\label{chainofrelations} x\odot 1=f_0\leq' f_0'\cong f_1\leq' f_1'\cong\dots\cong f_m\leq f_m'\cong f_{m+1}=y\odot 1.
\end{equation}

We will see that the elements $f_i,f_i'$ can be chosen of a particular form. To this end, let us introduce some new notation. We denote by $\N'$ the set of non compact natural numbers in $Z$. Then, given a prime $p>0$, we consider $\N'[1/p]$ the dense subset of $(0,\infty)\subset Z$ and denote $Z_p=\N_0\sqcup \N'[1/p]\subset Z$ which can be seen as a discrete version of $Z$.

First suppose $f_i,f_i'\in \N[M^\times\times(\N\sqcup \N')^\times]$. 
In this case we will prove that $x\leq y$, and it will be enough to prove  that whenever $f\leq'g$, $f\rightarrow g$ or $f\leftarrow g$ with $f=\sum_{i=1}^n x_i\odot n_i$, $g=\sum_{j=1}^m y_j\odot m_j$,  $x_i,y_j\in V$ and $n_i,m_j\in \N\sqcup \N'$, then 
$\sum x_in_i\leq \sum y_jm_j$.  

Suppose $f\leq'g$, then $m=n$, $x_i\leq y_i$, $n_i\leq m_j$ and thus, clearly $\sum_i n_i\cdot x_i\leq \sum_i m_i\cdot y_i$.
If $f\rightarrow g$. Then $x_i=\sum_j x_{ij}$, $n_i=\sum_k n_{ik}$ and $g=\sum_{i,j,k}x_{ij}\odot n_{ik}$, and it is clear that $\sum_i x_in_i=\sum_{i,j,k} x_{ij}n_{ik}$ as desired. Finally, the case $f\leftarrow g$ is similar.

We claim that for any prime $p>0$, $f_i,f_i'$ can be chosen in $\N[M^\times\times Z_p^\times]$. Then, multiplying \eqref{chainofrelations} by a certain power $p^{n_0}$, we will obtain   
a chain \[(p^{n_0}x)\odot 1=h_0 \leq' h_0'\cong h_1\leq' h_1'\cong\dots\cong h_m\leq h_m'\cong h_{m+1}=(p^{n_0}y)\odot 1,\]
where each $h_i\in \N[M^\times\times(\N\sqcup \N')^\times]$, and by the previous argument $p^{n_0}x\leq p^{n_0}y$. Doing this for two different primes $p,q$ yelds $x\leq_p y$ as desired.  

Let us prove the claim. 
Set $p>0$ any prime. First, note that by adding trivial relations of the form $\leq'$, one can assume that in \eqref{chainofrelations} we have a concatenation of chains of the form $f_i\leq'f_i'\rightarrow f_{i+1}$ or $f_i\leq'f_i'\leftarrow f_{i+1}$.

Given $f,g\in \N[M^\times \times Z^\times]$ we will write $f\prec g$ provided $f=\sum_{i=1}^{n} x_i\odot \alpha_i$, $g=\sum_{i=1}^{n} y_i\odot \beta_i$ with $x_i\leq y_i$ and $\alpha_i\ll \beta_i$.

We set $g_0=x\odot 1\in \N[M^\times \times Z_p^\times]$, for which we have $g_0\prec f_0\leq'f_0'\rightleftarrows f_1$, and we
will inductively modify \eqref{chainofrelations}, by replacing $g_n\prec f_n\leq'f_n' \rightleftarrows f_{n+1}$, where $g_n\in \N[M^\times \times Z_p^\times]$, with $g_n\leq' g_n'
\rightleftarrows g_{n+1}\prec f_{n+1}$ with $g_n',g_{n+1}\in \N[M^\times \times Z_p^\times]$.
\[ \left(g_n\prec f_n\leq'f_n' \rightleftarrows f_{n+1}\right) \ \rightsquigarrow \  
\left( g_n\leq' g_n'
\rightleftarrows g_{n+1}\prec f_{n+1}\right)\]

We first consider the case $g_n\prec f_n\leq'f_n\rightarrow f_{n+1}$ with $g_n\in \N[M^\times \times Z_p^\times]$. 

Write $g_n=\sum_{i=1}^r x_i\odot \alpha_i$, $f_n=\sum_{i=1}^r x'_i\odot \alpha'_i$ and $f_n'=\sum_{i=1}^r y_i\odot \beta_i$, $x_i\leq x_i'\leq y_i$ and $\alpha_i\ll \alpha_i'\leq \beta_i$. Since $f_n\rightarrow f_{n+1}$ we have 
$y_i=\sum_{j=1}^{s_i} y_{ij}$, $\beta_i=\sum_{k=1}^{t_i}\beta_{ij}$, and $f_{n+1}=\sum_{i=1}^{r}\sum_{j,k}y_{ij}\odot \beta_{ik}$. 

Now, if $\beta_i\in \N$ (i.e. if it is compact), set $\beta'_i=\beta_i$, note that $\beta_{ik}\in \N$ (also compact) and set $\beta_{ik}'=\beta_{ik}$. If $\beta_i\not\in \N$, then as $\alpha_i\ll   \alpha_i'\leq \beta_i$, we have $\alpha_i<\beta_i=\sum_{k=1}^{t_i}\beta_{ik}$, and we can choose $\beta_{ik}'\in \N'[1/p]$, with $\beta_{ik}'<\beta_{ik}$ and such that $\alpha_i\leq \beta_i':=\sum_{k=1}^{t_i}\beta_{ik}'$.

Define $g'_n:=\sum_{i=1}^r y_i\odot \beta_i'$, $g_{n+1}:=\sum_{i,j,k} y_{ij}\odot\beta_{ik}'$ and we are done.

Now consider the case that $g_n\prec f_n\leq'f_n'\leftarrow f_{n+1}$ with $g_n\in \N[M^\times \times Z_p^\times]$. 

We now change notation and write, since $f_n\leftarrow f_{n+1}$, $f_{n+1}=\sum_{i=1}^r x_i\odot \alpha_i$, with $x_i=\sum_{j=1}^{s_i}z_{ij}$, $\alpha_i=\sum_{k=1}^{t_i}\alpha_{ik}$ and such that $f'_n=\sum_i\sum_{j,k} z_{ij}\odot \alpha_{ik}$. Also, $f_n=\sum_{i,j,k} w'_{ij}\odot \beta'_{ik}$ and $g_n=\sum_{i,j,k} w_{ij}\odot \beta_{ik}$, with $w_{ij}\leq w'_{ij}\leq z_{ij}$ and $\beta_{ik}\ll \beta_{ik}'\leq\alpha_{ij}$.

If $\alpha_{ik}\in \N$, set $\alpha'_{ik}=\alpha_{ik}$. If $\alpha_{ik}\not\in\N$, since $\beta_{ik}\ll \beta_{ik}'\leq\alpha_{ik}$, we have $\beta_{ik}<\alpha_{ik}$ and we can choose $\alpha_{ik}'\in \N'[1/p]$ such that $\beta_{ik}\leq \alpha'_{ik}<\alpha_{ik}$. Setting  $\alpha_i':=\sum_{i,k}\alpha'_{ik}\in \N'[1/p]\sqcup \N$, $g_n':=\sum_{i,j,k}z_{ij}\odot \alpha_{ik}$ and $g_{n+1}:=\sum_ix_i\odot \alpha_i'$ satisfy the conditions.

By induction, and since $f_{m+1}=y\odot 1$ is already of the desired form,  we have $x\odot 1=g_0\leq' g_0'\cong g_1\leq' g_1'\cong\dots\cong g_m\leq g_m'\cong g_{m+1}=y\odot 1$, where $g_i,g_i'\in \N[M^\times\times Z_p^\times]$, which proves the claim and completes the proof.
\end{proof}

Note that the converse of (i) is not allways true as the following example testifies.

\begin{exa}\label{ex:noconverse}
Let $N$ be a simple, positively ordered semigroup. Suppose that $N$ has two elements $a,b$ such that $a\not\leq b$, but there is some $n\in \mathbb{N}$ with $(n+1)a\leq nb$ (i.e. $N$ is not almost unperforated). Consider the elements $x=(a,1)$, $y=(b,1)\in M: =N\oplus\N $. Clearly, $x\not\leq y$ in $M$, and $(n+1)x\not\leq ny$ for any $n\in \mathbb{N}$. Finally, as $M\otimes Z\cong N\otimes Z\oplus Z$, we see that indeed in this case $x\otimes 1\leq y\otimes 1$.
\end{exa}

We now turn into the question of analysing when $M\otimes 1$ is a nearly unperforated semigroup. 

\begin{thm}
Let $M$ be a positively ordered semigroup. Consider the following conditions:
\begin{enumerate}[{\rm (i)}]
\item $M$ is nearly unperforated.
\item The map $M\to M\otimes 1$ is an order-embedding. 
\item $M$ is almost unperforated.
\end{enumerate}
Then {\rm (i)}$\implies${\rm (ii)}$\implies${\rm (iii)}. If $M$ is algebraically ordered and cancellative, {\rm (iii)}$\implies${\rm (i)} and all three conditions are equivalent.
\end{thm}
\begin{proof}
The first part follows from \autoref{lma:orderStensorZ}. If $M$ is algebraically ordered, cancellative, and almost unperforated, then it is also nearly unperforated by \autoref{prp:equivnuau}.
\end{proof}

\begin{rmks}
Using results from \cite{APT14} that we will discuss in the next section, one can show that $M\otimes 1$ is always almost unperforated. However, as the order is not necessarily algebraic, it does not appear possible to immediately conclude near unperforation. 

If $M$ is nearly unperforated and separative the map $M\to M\otimes 1$ is not only an order-embedding, but an order-isomorphism. Indeed, if $x\otimes 1=y\otimes 1$, then by \autoref{lma:orderStensorZ} we have $nx=ny$ and $(n+1)x=(n+1)y$ for some $n\in\N$. Applying separativity we obtain $x=y$. (Of course, assuming only separativity $M$ and $M\otimes 1$ are isomorphic as semigroups.)
\end{rmks}

A semigroup $M$ is said to be a \emph{refinement} semigroup if whenever $x_1+x_2=y_1+y_2$ in $M$, there are elements $z_{ij}\in M$ for $i=1,2$ such that $x_i=z_{i1}+z_{i2}$ and $y_j=z_{1j}+z_{2j}$ for $i,j=1,2$. These semigroups are relevant since $\V(A)$ is always a refinement semigroup for any \ca{} $A$ of real rank zero. 

\begin{dfn}
Let $M$ be a positively ordered semigroup. For elements $x$, $y\in M$, write $x\leq_d y$ provided there are elements $x_i$, $y_i\in M$, for $i=1,2$ such that $x=x_1+x_2$, $y=y_1+y_2$, and $x_1\leq y_1$, while $x_2<_s y_2$.
\end{dfn}

\begin{lma}
\label{lma:comparisonrelns}
Let $M$ be a positively ordered semigroup. Consider the following conditions, for $x$,$y\in M$:
\[ \mathrm{ (i) }\  x\leq_s y\qquad   \mathrm{ (ii) }\  x\leq_d y \qquad  \mathrm{ (iii) }\ x\leq_p y.\]  
Then {\rm (i)}$\implies${\rm (ii)}$\implies${\rm (iii)}. If $M$ is a refinement, algebraically ordered and cancellative semigroup, {\rm (iii)}$\implies${\rm (ii)}. If $M$ is algebraically ordered, cancellative and simple, then all three conditions are equivalent.
\end{lma}
\begin{proof}
(i)$\implies$ (ii): This is trivial, noting that we always have $0<_s 0$.

(ii)$\implies$ (iii): If $x\leq_d y$, we have $x=x_1+x_2$, $y=y_1+y_2$ with $x_1\leq y_1$ and $(n+1)x_2\leq ny_2$ for some $n$. Then also $nx_2\leq ny_2$ and $(n+1)x_2\leq (n+1)y_2$, so that $x_1\leq_p y_1$ and $x_2\leq_p y_2$. Thus $x\leq_p y$.

Now assume (iii) and also that $M$ is a refinement, algebraically ordered, cancellative semigroup. Since $x\leq_p y$, we have that $nx\leq ny$ and $(n+1)x\leq (n+1)y$ for some $n$. We may find then $z$ and $w\in M$ such that $nx+z=ny$ and $(n+1)x+w=(n+1)y$. As in \autoref{prp:equivnuau}, cancellation implies that $(n+1)z=nw$, and thus $z<_s w$. Also
\[
x+nx+w=ny+y=nx+z+y\,,
\]
whence $x+w=y+z$. By applying refinement to this equality, we find elements $x_{ij}\in M$ such that $x=x_{11}+x_{12}$, $y=x_{11}+x_{21}$, while $z=x_{12}+x_{22}$ and $w=x_{21}+x_{22}$. Put $x_1=y_1=x_{11}$, $x_2=x_{12}$, and $y_2=x_{21}$. Since $x_{12}+x_{22}=z<_s w=x_{21}+x_{22}$, it follows from cancellation that $x_2<_s y_2$.

Finally assume that $M$ is simple, algebraically ordered, and cancellative, and assume that $x\leq_p y$. If $x\not\leq y$, there is by cancellation a natural number $n$ such that $nx<ny$. Now find a non-zero element $z$ such that $nx+z=ny$. Since $M$ is simple and $z$ is non-zero, there is $k$ such that $(k+1)x\leq ky$, so that $x<_sy$.
\end{proof}

\begin{prp}
\label{prp:charactofxotimes1}
Let $M$ be an algebraically ordered, cancellative semigroup. Suppose that $M$ is either simple or else that $M$ is a refinement semigroup. Then $x\otimes 1\leq y\otimes 1$ in $M\otimes Z$ if, and only if, $x\leq_p y$.
\end{prp}
\begin{proof}
By \autoref{lma:orderStensorZ} we only need to check the `only if' direction. Thus, assume $x\leq_p y$. 

If $M$ is simple, then \autoref{lma:comparisonrelns} shows that $x\leq_s y$, and we may apply \autoref{lma:orderStensorZ} to conclude that $x\otimes 1\leq y\otimes 1$.

If $M$ satisfies the Riesz decomposition property, then again by \autoref{lma:comparisonrelns}, we see that $x\leq_d y$, so that $x=x_1+x_2$ and $y=y_1+y_2$ with $x_1\leq y_1$ and $x_2<_s y_2$. Using \autoref{lma:orderStensorZ} at the second step we obtain
\[
x\otimes 1=x_1\otimes 1+x_2\otimes 1\leq y_1\otimes 1+y_2\otimes 1=y\otimes 1\,,
\]
as desired.
\end{proof}

 \begin{prp}
\label{prp:charactofnunperf}
Let $M$ be an algebraically ordered, cancellative semigroup. Then, the following conditions are equivalent:
\begin{enumerate}[{\rm (i)}]
\item $M\otimes 1$ is order-cancellative;
\item $M\otimes 1$ is order-separative;
\item $M\otimes 1$ is nearly unperforated;
\item $x\leq_p y\implies x\otimes 1\leq y\otimes 1$ in $M\otimes 1$.
\end{enumerate}
\end{prp}
\begin{proof}
It is clear that (i)$\implies$ (ii).

Let us show that (ii)$\implies$ (iv): Assume that $x\leq_p y$. Then $nx+z= ny$  and $(n+1)x+w= (n+1)y$ for some $n$, and cancellation implies that $(n+1)z=nw$ (as in the proof of \autoref{lma:comparisonrelns}), so $z\otimes 1\leq w\otimes 1$ by \autoref{lma:orderStensorZ}. Using this at the second step, we obtain:
\[
x\otimes 1+ny\otimes 1=(n+1)x\otimes 1+z\otimes 1\leq (n+1)x\otimes 1+w\otimes 1=y\otimes 1+ny\otimes 1\,.
\]
Since $M\otimes 1$ is assumed to be separative, we obtain $x\otimes 1\leq y\otimes 1$.

(iv)$\implies$ (iii): This is clear using \autoref{lma:orderStensorZ}.

(iii)$\implies$ (i): Suppose that $x\otimes 1+z\otimes 1\leq y\otimes 1+z\otimes 1$. Using \autoref{lma:orderStensorZ} and cancellation on $M$ we obtain $x\leq_p y$. This implies $x\otimes 1\leq_p y\otimes 1$ and by (iii) we have $x\otimes 1\leq y\otimes 1$, as desired.
\end{proof}

\begin{cor}
\label{cor:charactofnunperf}
Let $M$ be an algebraically ordered, cancellative semigroup. Suppose that $M$ is either simple or else that $M$ is a refinement semigroup. Then $M\otimes 1$ is order-cancellative and nearly unperforated.
\end{cor}
\begin{proof}
By \autoref{prp:charactofxotimes1}, $x\leq_p y\iff x\otimes 1\leq y\otimes 1$. The result then follows from \autoref{prp:charactofnunperf}.
\end{proof}

To close this section, we relate these results to the constructions in \autoref{sec:cones}:

\begin{prp}
 \label{prp:auequals2plus}
 Let $M$ be an algebraically ordered, cancellative semigroup. Suppose that $M\otimes 1$ is order-cancellative. Then
$\Mau\cong\Gr(M\otimes 1)^{++}$.
\end{prp}
\begin{proof}
 We have to define a semigroup isomorphism $\varphi\colon\Mau\to\Gr(M\otimes 1)^{++}$.
 
 Denote by $\iota_M\colon M\to\Gr(M)^+$ the Grothendieck map for $M$, and by $\iota_{M\otimes 1}^+\colon M\otimes 1\to\Gr(M\otimes 1)^{++}$ the corresponding map for $M\otimes 1$.
 
 Given $x\in\Mau$, write $x=\iota(a)-\iota(b)$, for $a,b\in M$. Since there is $n\in\N$ such that $nx$ belongs to $\Gr(M)^+=\iota(M)$ and $M$ is cancellative, we have $nb\leq na$ in $M$. Likewise, $(n+1)b\leq (n+1)a$ as $(n+1)x\in\Gr(M)^+$. Thus $a\leq_pb$. Now \autoref{prp:charactofnunperf} implies $b\otimes 1\leq a\otimes 1$ in $M\otimes 1$. Put $\varphi(x)=\iota_{M\otimes 1}^+(a\otimes 1)-\iota_{M\otimes 1}^+(b\otimes 1)$.
 
 In order to verify that $\varphi$ is well defined, suppose that $x=\iota(a)-\iota(b)=\iota(c)-\iota(d)$. Then $a+d=b+c$ in $M$ and thus $a\otimes 1+d\otimes 1=b\otimes 1+c\otimes 1$ in $M\otimes 1$. It follows from this that $\iota_{M\otimes 1}^+(a\otimes 1)-\iota_{M\otimes 1}^+(b\otimes 1)=\iota_{M\otimes 1}^+(c\otimes 1)-\iota_{M\otimes 1}^+(d\otimes 1)$.
 
 It is clear that $\varphi$ is additive. To check injectivity, retain notation from the previous paragraph, and assume $\varphi(\iota(a)-\iota(b))=\varphi(\iota(c)-\iota(d))$. Since $M\otimes 1$ is order-cancellative, this means that $(a+d)\otimes 1=(c+b)\otimes 1$. Now \autoref{lma:orderStensorZ} combined with cancellation imply $a+d=c+b$. Thus $\iota(a)-\iota(b)=\iota(c)-\iota(d)$.
 
 Finally, given $y=\iota_{M\otimes 1}^+(a\otimes 1)-\iota_{M\otimes 1}^+(b\otimes 1)\in\Gr(M\otimes 1)^{++}$ with $b\otimes 1\leq a\otimes 1$, apply \autoref{lma:orderStensorZ} to obtain $b\leq_p a$ and thus $x:=\iota(a)-\iota(b)$ is an element in $\Mau$ with $\varphi(x)=y$.
\end{proof}

\begin{cor}
\label{cor:VAtensorZ} Let $A$ be a \ca{} containing a full projection. Suppose that $A$ has real rank zero, cancellation into ideals, and $A\otimes\ZZ$ is residually stably finite. Then $V(A\otimes\ZZ)\cong\Gr(\V(A)\otimes 1_Z)^{++}$.
\end{cor}
\begin{proof}
This is an application of the proof of \autoref{cor:isthatyou?} and \autoref{prp:auequals2plus}.
\end{proof}

\section{Applications to Cu-semigroups}
\label{sec:ApplicationsToCu}

In this section we apply our results to the study of the structure of algebraic $\CatCu$-semigroups tensored with the Cuntz semigroup of the Jiang-Su algebra $\ZZ$. As a consequence, we uncover a somewhat surprising and new connection between the axiom of almost algebraic order \axiomO{5} and almost unperforation.

We start by recalling the construction of the tensor product in the category $\CatCu$, as well as its main features, since this will be used in the remaining of the paper.

\begin{pgr}[Tensor products in the category $\CatCu$ and algebraic $\CatCu$-semigroups]
\label{pgr:Cu}
Given a positively ordered semigroup $S$, an auxiliary relation on $S$ (see \cite{Compendium}) is a binary relation $\prec$ that is weaker than the order, such that $0\prec x$ for any $x\in S$, and satisfying
\[
x\leq y\prec z\leq t\implies x\prec t\,,
\]
for any $x,y,z,t\in S$. 

As immediate examples of auxiliary relations, we have the order $\leq$ in any positively ordered semigroup, and also
the relation $\ll$ of compact contaniment.

In \cite{APT14}, a category $\CatPreW$ is defined for positively ordered semigroups that have an auxiliary relation. This category sets up a convenient framework to manufacture tensor products, that can then be brought back to $\CatCu$. The axioms defining this category parallel those used to define the category $\CatCu$, and will not be repeated here as they will not be explicitely needed. We mention that $\CatCu$ is a subcategory of $\CatPreW$, and that any positively ordered semigroup $S$ can be viewed as a $\CatPreW$-semigroup, by taking its order as an auxiliary relation.

The category $\CatPreW$ admits a tensor product construction, as follows (see \cite[Section 6.2]{APT14} for details): given $S$ and $T$ objects in $\CatPreW$ with their respective auxiliary relations, then $S\otimes_\CatPreW T$ is the ordinary partially ordered semigroup tensor product, as described in \ref{pgr:tensprod}, equipped with a relation $\prec$ defined as follows, for $x,y\in S\otimes T$:
\[
x\prec y \text{ if whenever }y=\sum_i s_i\otimes t_i,\text{ there are }s_i'\prec s_i, t_i'\prec t_i \text{ with }x\leq\sum_is_i'\otimes t_i'\,. 
\]
This of course applies to $\CatCu$-semigroups as well, but is not enough to give the tensor product a $\CatCu$-semigroup structure. One needs to apply a completion functor $\gamma\colon\CatPreW\to\CatCu$, as in \cite[Section 3.1]{APT14}, showing that $\CatCu$ is in fact a reflective subcategory of $\CatPreW$, and thus $\gamma(S)\cong S$ precisely when $S$ is a $\CatCu$-semigroup. Therefore, for $\CatCu$-semigroups $S$ and $T$, one has $S\otimes_\CatCu T=\gamma(S\otimes_\CatPreW T)$ (see \cite[Theorem 6.3.3]{APT14}). If we start with $\CatPreW$-semigroups $S$ and $T$, then it is proved in \cite[Theorem 6.3.5]{APT14} that $\gamma(S\otimes_\CatPreW T)\cong\gamma(S)\otimes_\CatCu\gamma(T)$. For a $\CatPreW$-semigroup $S$ with an auxiliary relation $\prec$, we will sometimes also write $\gamma(S,\prec)$ to ensure clarity.

From the completion point of view above, a $\CatCu$-semigroup $S$ is algebraic precisely when $S=\gamma(S_c)$. Hence, the tensor product of an algebraic $\CatCu$-semigroup with any other $\CatCu$-semigroup $T$ can always be thought as $S\otimes_\CatCu T\cong\gamma(S_c\otimes_\CatPreW T)$. We shall be using this observation below.

As proved in \cite[Theorem 5.5.8]{APT14}, algebraic $\CatCu$-semigroups are relevant since various structural properties of $S$ are determined by corresponding properties of $S_c$. For example, $S$ satisfies \axiomO{5} if, and only if, $S_c$ is algebraically ordered. Also, $S$ has weak cancellation if, and only if, $S_c$ is a cancellative semigroup. If $S_c$ satisfies the Riesz decomposition property, then $S$ satisfies the so called axiom of almost Riesz decomposition:

Recall that a $\Cu$-semigroup $S$ is said to satify axiom \axiomO{6} (the axiom of \emph{almost Riesz decomposition}) if, whenever $x'\ll x\leq y+z$, there exist $y_0\leq x,y$ and $z_0\leq x,z$ such that $x'\leq y_0+z_0$. This axiom was introduced by L.~Robert in \cite{Rob13Cone}, where it was proved that for any \ca{} $A$, its Cuntz semigroup $\Cu(A)$ satisfies \axiomO{6} (see \cite[Proposition 5.1.1]{Rob13Cone}).

The Cuntz semigroup $\Cu(A)$ for a \ca{} $A$ of real rank zero is the typical example to bear in mind of an algebraic semigroup.
\end{pgr}

We recall three of the questions that motivate this paper:

\begin{prbl}
\label{pr:revisited}
\begin{enumerate}[{\rm (i)}]
\item If $S$ is an algebraic $\CatCu$-semigroup, does it follow that $S\otimes_\CatCu Z$ is also algebraic?
\item Given a $\CatCu$-semigroup $S$, characterize when $x\otimes 1\leq y\otimes 1$ in $S\otimes_\CatCu Z$.
\item When do axioms \axiomO{5}, \axiomO{6} or the condition of weak cancellation pass from $S$ to $S\otimes_\CatCu Z$?
\end{enumerate}
\end{prbl}

We shall mostly focus on the algebraic case, since then results in previous sections will become available. Thus in the next two lemmas we  analyse the order in the tensor product $M\otimes_\CatPreW Z$ when $M$ is an almost divisible semigroup (see \autoref{dfn:div}).

\begin{lma}
\label{lma:StensorZ}
Let $M$ be an almost divisible positively ordered semigroup. Let $x$ be an element in $M$, and $s<t$ in $(0,\infty)$. Then, there exist $n\in\N$, $y\in M$ such that $x\otimes s\leq ny\otimes 1_Z\leq x\otimes t$ in $M\otimes_{\CatPreW} Z$.
\end{lma}
\begin{proof}
If $t-s\geq 1$, then there is $n$ such that $s\leq n\cdot 1'\leq n\cdot 1_Z\leq t$, and so in this case $x\otimes s\leq nx\otimes 1_Z\leq x\otimes t$.

Suppose that $t-s<1$. Then, we may choose $L$ such that $L>\frac{1+t}{t-s}$. Then $(L-1)t-Ls>1$, hence we may find $n\in\N$ such that $Ls<n<(L-1)t$. Since $M$ is almost divisible, there exists $y\in M$ such that $(L-1)y\leq x\leq Ly$. Then:
\[
x\otimes s\leq Ly\otimes s=y\otimes Ls\leq y\otimes n1_Z\leq y\otimes (L-1)t=(L-1)y\otimes t\leq x\otimes t\,.
\]
\end{proof}

\begin{lma}
 \label{lma:precStensorZ}
 Let $M$ be an almost divisible positively ordered semigroup. Then $a\prec b$ in $M\otimes_{\CatPreW}Z$ if and only if there exists $c\prec c$ in $M\otimes_{\CatPreW}Z$ such that $a\prec c\prec b$.
 
Further, an element $a\in M\otimes_\CatPreW Z$ satisfies $a\prec a$ if and only if there is $x\in M$ such that $a=x\otimes 1_Z$.
\end{lma}
\begin{proof}
 Write $b=\sum_i x_i\otimes t_i$. Since $a\prec b$, there exist, for each $i$, elements $x_i'\leq x_i$ and $s_i\ll t_i$ such that $a\leq\sum_i x_i'\otimes s_i$. Write $Z=\N\sqcup (0,\infty]$, and put $I=\{i\mid s_i\in\N\subset Z\}$, and $J=\{i\mid s_i\in (0,\infty]\subset Z\}$. There is nothing to prove if $J=\emptyset$, as in this case we may take $c=\sum_{i\in I} x_i'\otimes s_i$.
 
Assume then $J\neq\emptyset$. For each $i\in J$, we have $s_i<t_i\leq\infty$, and we may use \autoref{lma:StensorZ} to find $z_i\in M$, $n_i\in\N$ such that
 \[
 x_i'\otimes s_i\leq n_iz_i\otimes 1_Z\leq x_i'\otimes t_i\leq x_i\otimes t_i\,.
 \]
 Then, let $c=\sum_{i\in I} x_i'\otimes s_i+\sum_{i\in J} n_iz_i\otimes 1_Z$, and clearly  $a\prec c\prec b$, with $c\prec c$.
 
If $a\prec a$, we may apply the first part of the proof with $a=b$ and conclude that $a=x\otimes 1_Z$ for some $x\in M$, as desired.
\end{proof}

In the result below we answer \autoref{pr:revisited} (i) in the affirmative, assuming further that the algebraic $\CatCu$-semigroup $S$ is almost divisible. In this case, as shown in \cite[Lemma 7.3.6]{APT14}, this is equivalent to the demand that the subsemigroup $S_c$ of compact elements is almost divisible. We remark that this is not an unreasonable condition. Indeed, if $S$ is weakly cancellative, and satisfies \axiomO{5}, \axiomO{6}, then $S_c$ is a cancellative refinement semigroup. Thus, by \autoref{lem:almostdivvsweakdiv}, $S$ is almost divisible if and only if $S_c$ is weakly divisible, and the latter condition will hold if and only if no non-zero quotient of $S_c$ has an ideal isomorphic to $\N$ (by \cite[Proposition 2.6]{OrtPerRorTAMS2011}). This translates into \ca{s} as follows: if $A$ is a \ca{} of real rank zero and stable rank one, then $\Cu(A)$ is almost divisible if and only if $\V(A)$ is weakly divisible. By \cite[Theorem 5.8]{PerRorJFunct04}, this is equivalent to the existence of no representation of $A$ on a Hilbert space whose image meets the compact operators non-trivially.

\begin{thm}
\label{thm:StensorZ}
Let $S$ be an algebraic, almost divisible $\CatCu$-semigroup. Then the semigroup $S\otimes_\CatCu Z$ is also algebraic with $(S\otimes_\CatCu Z)_c=S_c\otimes 1_Z$.
\end{thm}
\begin{proof}
Since $S$ is almost divisible as a $\CatCu$-semigroup, we have that $S_c$ is almost divisible as a positively ordered semigroup (\cite[Lemma 7.3.6]{APT14}).

Since $S=\gamma(S_c,\leq)$, and $Z=\gamma(Z,\ll)$ we have, using \cite[Theorem 6.3.5]{APT14} (see also the comments in \ref{pgr:Cu}), that
\[
S\otimes_\CatCu Z\cong \gamma(S_c\otimes_\CatPreW Z)\,.
\]
To ease the notation, we shall then identify $S\otimes_\CatCu Z$ with $\gamma(S_c\otimes_\CatPreW Z)$, and we will also use $\alpha\colon S_c\otimes_\CatPreW Z\to \gamma(S_c\otimes_\CatPreW Z)$ to denote the $\CatPreW$-map of the $\CatCu$-completion (see \cite[Definition 3.1.7 and Proposition 3.1.6]{APT14}). We will denote by $\prec$ the auxiliary relation in $S_c\otimes_\CatPreW Z$ (see \ref{pgr:Cu} and \cite[Definition 6.2.9]{APT14}).

Let $a\in S\otimes_\CatCu Z$. Then, by \cite[Theorem 3.1.8]{APT14}, there exists a sequence $(a_n)$ in $S_c\otimes_\CatPreW Z$ such that $a_n\prec a_{n+1}$ and $a=\sup \alpha(a_n)$ (and, in fact, $\alpha(a_n)\ll \alpha(a_{n+1})$). By \autoref{lma:precStensorZ}, we can find, for each $n$, elements $c_n\in S_c\otimes_\CatPreW Z$ such that $a_n\prec c_n\prec c_n\prec a_{n+1}$. This means that $\alpha(a_n)\ll\alpha(c_n)\ll\alpha(a_{n+1})$, with $\alpha(c_n)$ compact for each $n$. Thus compact elements are dense and hence $S\otimes_\CatCu Z$ is algebraic.

This argument also shows that, in fact, $(S\otimes Z)_c =\{\alpha(x\otimes 1)\mid x\in S_c\}=\alpha(S_c\otimes 1_Z)\cong S_c\otimes 1_Z$.
\end{proof}

We next characterize, for a wide class of algebraic $\CatCu$-semigroups $S$, when $x\otimes 1\leq y\otimes 1$ in $S\otimes_\CatCu Z$, thus giving a partial answer to \autoref{pr:revisited} (ii). We also address part of \autoref{pr:revisited} (iii).

\begin{thm}
\label{thm:pr2}
Let $S$ be an algebraic, almost divisible $\CatCu$-semigroup. Suppose that $S$ has weak cancellation and satisfies \axiomO{5}, \axiomO{6}. Then:
\begin{enumerate}[{\rm (i)}]
\item $S\otimes_\CatCu Z$ is nearly unperforated and weakly cancellative.
\item For $x,y\in S$, $x\otimes 1_Z\leq y\otimes 1_Z$ in $S\otimes_\CatCu Z$ if and only $x'\leq_p y$ for all $x'\ll x$.
\end{enumerate}
\end{thm}
\begin{proof}
Retain the notation in the proof of \autoref{thm:StensorZ}.

(i). We know that $S\otimes_\CatCu Z$ is algebraic. Hence, in order to check that $S\otimes_\CatCu Z$ is weakly cancellative and nearly unperforated, we need to show that the subsemigroup of compact elements is cancellative and nearly unperforated and then use \cite[Theorem 5.5.8, Lemma 5.6.16]{APT14}. 

Our assumptions imply that $S_c$ is an algebraically ordered, cancellative, refinement semigroup. Thus, by \autoref{cor:charactofnunperf}, $S_c\otimes 1_Z$ is cancellative and nearly unperforated. Finally, we have $(S\otimes Z)_c \cong S_c\otimes 1_Z$ from \autoref{thm:StensorZ}.

(ii). Suppose first that $x'\leq_p y$ whenever $x'\ll x$. Write $x=\sup s_n$ where $(s_n)$ is a rapidly increasing sequence consisting of compact elements. For each $n$, we have $s_n\leq_p y$, and therefore $\alpha(s_n\otimes 1)\leq_p y\otimes 1$. By {\rm (i)}, $S\otimes_\CatCu Z$ is nearly unperforated, and thus $\alpha(s_n\otimes 1)\leq y\otimes 1$ for each $n$. Therefore $x\otimes 1\leq y\otimes 1$.

Conversely, suppose that $x\otimes 1\leq y\otimes 1$ and take $x'\in S$ with $x'\ll x$. Find a compact element $c$ in $S$ such that $x'\ll c \ll x$. Then we have $\alpha(c \otimes 1) \ll \alpha( c\otimes 1) \leq y\otimes 1$. Let $(t_m)_m$ be a sequence of compact elements such that $y=\sup t_m$. Thus $y\otimes 1=\sup \alpha(t_m\otimes 1)$. It follows that there is $k\in\N$ with $\alpha( c\otimes 1)\leq \alpha(t_k\otimes 1)$. Using \cite[Theorem 3.1.8 (1)]{APT14}, we get $c\otimes 1\leq t_k\otimes 1$. Now using \autoref{lma:orderStensorZ} we have $c\leq_p t_k\leq y$, and hence $x'\leq_p y$.
\end{proof}
\begin{rmk}
 A priori it is not clear that, even when $S$ is algebraic, almost divisible and satisfies \axiomO{5}, \axiomO{6}, the semigroup $S\otimes_\CatCu Z$ will satisfy \axiomO{5} (more about this in the results below). If that were the case, then near unperforation as in \autoref{thm:pr2} (i) would be automatic by virtue of \cite[Proposition 5.6.12]{APT14}, as $S\otimes_\CatCu Z$ is always almost unperforated. 
 
Notice also that, under the assumptions of \autoref{thm:pr2}, the map $S\to S\otimes_\CatCu Z$ given by $x\mapsto x\otimes 1_Z$ is an order-embedding if and only if $S$ is nearly unperforated, if and only it is an order-isomorphism. This uses statement (i) in the said theorem, and \cite[Theorem 7.3.11, Theorem 7.1.12]{APT14}.
\end{rmk}

\begin{cor}
Let $A$ be a \ca{} with real rank zero.
\begin{enumerate}[{\rm (i)}]
\item If $\Cu(A)$ is almost divisible, then $\Cu(A)\otimes_\CatCu Z$ is an algebraic $\CatCu$-semigroup.
\item Furthermore, if $A$ has stable rank one, then $\Cu(A)\otimes_\CatCu Z$ has weak cancellation and is nearly unperforated.
\end{enumerate}
\end{cor}
\begin{proof}
This follows from \autoref{thm:StensorZ}.
\end{proof}

We now turn our attention to \autoref{pr:revisited} (iii). More concretely, we look into the question of whether \axiomO{5} passes from an algebraic $\CatCu$-semigroup $S$ to $S\otimes_\CatCu Z$.

\begin{thm}
\label{thm:StensorZO5}
Let $S$ be an algebraic, almost divisible $\CatCu$-semigroup. Suppose that the subsemigroup $S_c$ of compact elements is separative. Consider the following conditions:
\begin{enumerate}[{\rm (i)}]
\item $S\otimes_\CatCu Z$ satisfies \axiomO{5}.
\item $S_c$ is almost unperforated.
\item $S$ is almost unperforated.
\item $S\cong S\otimes_\CatCu Z$ as $\CatCu$-semigroups.
\end{enumerate}
Then {\rm (i)}$\implies${\rm (ii)}$\implies${\rm (iii)}$\implies${\rm (iv)}. If further $S$ satisfies \axiomO{5}, then all conditions are equivalent.
\end{thm}
\begin{proof}
(i)$\implies$ (ii):
By \autoref{thm:StensorZ}, $S\otimes_\CatCu Z$ is algebraic, and the subsemigroup of compact elements is identified with $\{x\otimes 1_Z\mid x\in S_c\}$. Suppose that $(n+1)\leq ny$ in $S_c$. Then, by \autoref{lma:orderStensorZ}, we get $x\otimes 1\leq y\otimes 1$ in $S\otimes_\CatCu Z$. By assumption, $S\otimes_\CatCu Z$ satisfies \axiomO{5}, and hence the subsemigroup of compact elements is algebraically ordered, by \cite[Theorem 5.5.8]{APT14}. There exists then $z\in S_c$ such that $x\otimes 1+z\otimes 1=y\otimes 1$, and then $(x+z)\otimes 1=y\otimes 1$. By \autoref{lma:orderStensorZ} we obtain $n(x+z)=ny$ and $(n+1)(x+z)=(n+1)y$ for some $n$. Since $S_c$ is separative, we conclude $x+z=y$.

(ii)$\implies$(iii): This is immediate as $S$ is algebraic, hence the completion of $S_c$.

(iii)$\implies$(iv): Since $S$ is almost unperforated and almost divisible, \cite[Theorem 7.3.11]{APT14} applies to conclude $S\cong S\otimes_\CatCu Z$.

If $S$ satisfies \axiomO{5}, clearly (iv)$\implies$(i).
\end{proof}

Recall that for a stably finite \ca{} $A$, the subsemigroup of compact elements of $\Cu(A)$ can be identified with $\V(A)$. We now interpret \autoref{thm:StensorZO5} for \ca{s}. This connects the open problem of whether $\V(A)$ is always almost unperforated, for a \ca{} $A$ of real rank zero, with the fact that $\Cu(A)\otimes_\CatCu\Cu(\mathcal Z)$ satisfies \axiomO{5}.

\begin{cor}
\label{cor:StensorZca}
Let $A$ be a stably finite, separative, almost divisible \ca{} of real rank zero. Then the following conditions are equivalent:
\begin{enumerate}[{\rm (i)}]
\item $\Cu(A)\otimes_\CatCu \Cu(\mathcal Z)$ satisfies \axiomO{5}.
\item $\V(A)$ is almost unperforated.
\item $\Cu(A)$ is almost unperforated.
\item $\Cu(A)\cong\Cu(A)\otimes_\CatCu \Cu(\mathcal Z)$.
\end{enumerate}
\end{cor}
\begin{rmk}
Note that \autoref{cor:StensorZca} applies to a large class of stably finite \ca{s} of real rank zero. As already mentioned in the comments prior to \autoref{thm:StensorZ}, almost divisibility holds quite widely. Further, it is an open problem to determine whether all \ca{s} of real rank zero are separative (see \cite{AraGoodearlOMearaPardo98}).
\end{rmk}

If we restrict to the simple case, we further obtain the following:
\begin{cor}
\label{cor:StensorZsimple}
Let $A$ be a simple, non-type I, \ca{} of real rank zero and stable rank one. Then, the following conditions are equivalent:
\begin{enumerate}[{\rm (i)}]
\item $\Cu(A)\otimes_\CatCu \Cu(\mathcal Z)$ satisfies \axiomO{5}.
\item $\V(A)$ is almost unperforated.
\item $\Cu(A)$ is almost unperforated.
\item $\Cu(A)\cong\Cu(A)\otimes_\CatCu \Cu(\mathcal Z)$.
\item $\Cu(A)\cong\Cu(A\otimes\mathcal Z)$.
\end{enumerate}
\end{cor}
\begin{proof}
By \cite[Proposition 5.3]{PerRorJFunct04}, $A$ is weakly divisible, and thus $\Cu(A)$ is almost divisible (by \cite[Lemma 7.3.6]{APT14}). Since $A$ has stable rank one, it is certainly separative, an thus the equivalence between (i) and (iv) follows from \autoref{cor:StensorZca}.

Assume condition (iv). Then $\Cu(A)$ is almost unperforated and almost divisible, and we know that $\Cu(A)\cong V(A)^\times\sqcup L(F(\Cu(A)))$ (see, e.g. \cite[Theorem 2.6]{BroTomIMRN07}, \cite[Theorem 6.3]{AntBosPerIJM11}, \cite[Theorem 7.6.6]{APT14}). On the other hand, $\Cu(A\otimes\mathcal Z)\cong \V(A\otimes\mathcal Z)\sqcup L(F(\Cu(A\otimes\mathcal Z)))$. Since $\mathcal Z$ has a unique trace, we see that $F(\Cu(A\otimes\mathcal Z))\cong F(\Cu(A))$. 
Since $A$ is simple and has stable rank one, $A\otimes \ZZ$ also has stable rank one (\cite[Theorem 6.7]{Rordam04}). Since we also know that $\V(A)$ is almost unperforated, \autoref{thm:GJSimproved} coupled with \autoref{lma:Gau} imply that $\V(A\otimes\mathcal Z)\cong \V(A)$. Therefore $\Cu(A)\cong \Cu(A\otimes\mathcal Z)$. 

That (v) implies (iv) is clear because $\Cu(A\otimes\mathcal Z)$ is always almost unperforated and almost divisible.
\end{proof}

\end{document}